%% file: Density_Resonances.tex
\documentclass{amsart}

\usepackage{amsmath,graphics}
\usepackage{amssymb}
\usepackage{amsfonts}
\usepackage{bbm}
\usepackage{latexsym}
\usepackage{eucal}
\usepackage[dvips]{graphicx}
\usepackage{color}
\usepackage{enumerate}
\usepackage{mathrsfs}
\newtheorem{thm}{Theorem}[section]

\newtheorem{prop}[thm]{Proposition}

\newtheorem{lem}[thm]{Lemma}
\newtheorem{cor}[thm]{Corollary}
\newtheorem{rem}[thm]{Remark}

\input{makros}

\usepackage{hyperref}

\begin{document}
\title[Density of Resonances]{Density of Resonances for Covers of Schottky Surfaces}
\author[A.\@ Pohl]{Anke Pohl}
\address{Anke Pohl, University of Bremen, Department 3 -- Mathematics, Bibliothekstr.\@ 
5,  28359 Bremen, Germany}
\email{apohl@uni-bremen.de}

\author[L.\@ Soares]{Louis Soares}
\address{Louis Soares, Friedrich-Schiller-Universit\"at Jena, Institut f\"ur Mathematik, Ernst-Abbe-Platz 2, 07743 Jena, Germany
}
\email{louis.soares@uni-jena.de}


\maketitle
\begin{abstract} 
We investigate how bounds of resonance counting functions for Schottky surfaces behave under transitions to covering surfaces of finite degree. We consider the classical resonance counting function asking for the number of resonances in large (and growing) disks centered at the origin of $\CC$, as well as the (fractal) resonance counting function asking for the number of resonances in boxes near the axis of the critical exponent. For the former counting function we provide a transfer-operator-based proof that bounding constants can be chosen such that the transformation behavior under transition to covers is as for the Weyl law in the case of surfaces of finite area. For the latter counting function we deduce a bound in terms of the covering degree and the minimal length of a periodic geodesic on the covering surface. This yields an improved fractal Weyl upper bound. 
In the setting of Schottky surfaces, these estimates refine previous results due to Guillop\'{e}--Zworski  and Guillop\'{e}--Lin--Zworski. When applied to principal congruence covers, these results yield new estimates for the resonance counting functions in the level aspect, which have recently been investigated by Jakobson--Naud. The techniques used in this article are based on the thermodynamic formalism for $L$-functions (twisted Selberg zeta functions), and twisted transfer operators.
\end{abstract}

\section{Introduction and statement of main results}

The distribution, localization and asymptotics of resonances of the Laplacian of hyperbolic surfaces are of interest in several different areas of research, in particular in spectral theory, harmonic analysis, representation theory, number theory and mathematical physics. Classically, these questions were mainly studied for hyperbolic surfaces of finite area. Over the last decade, understanding these distributions as well for hyperbolic surfaces of infinite area has attracted some attention. The generalization of Selberg's $\tfrac{3}{16}$ Theorem by Bourgain--Gamburd--Sarnak~\cite{BGS} as well as the progress towards the Zaremba conjecture in number theory by Bourgain--Kontorovich~\cite{Bourgain_Kontorovich} are two examples of applications of recent distribution results for resonances of hyperbolic surfaces of infinite area. 

Whereas for hyperbolic surfaces of finite area the asymptotical distribution of resonances is fairly well understood by now, for hyperbolic surfaces of infinite area even some rather fundamental questions regarding these asymptotics are still open. With this article we contribute to the understanding of the distribution of the resonances of hyperbolic surfaces of infinite area.

Our main interest is to understand how bounds of resonance counting functions vary under the transition from a hyperbolic surface $X$ to a finite cover of $X$. 

To explain our motivation and results in more detail, let $X$ be a geometrically finite hyperbolic surface (of finite area or of infinite area), and let $\Delta_X$ denote the (positive) Laplacian on $X$. The resolvent 
\[
 R_X(s) \sceq \big( \Delta_X - s(1-s)\big)^{-1}\colon L^2(X) \to H^2(X)
\]
of $\Delta_X$ is defined for $s\in\CC$ with $\Rea s>1/2$ and $s(1-s)$ not being an $L^2$-eigenvalue of $\Delta_X$. It extends to a meromorphic family 
\[
 R_X(s) \colon L^2_{\text{comp}}(X) \to H^2_{\text{loc}}(X)
\]
on $\CC$ with poles of finite rank~\cite{Mazzeo_Melrose, GZ_upper_bounds}. The resonances of $X$ are the poles of this meromophic continuation. Let  $\mc R(X)$ denote the (multi-)set of resonances, repeated according to their multiplicities
\[
 m(s) \sceq \rank \Res_{t=s}\big(R_X(t)\big),
\]
where $\Res_{t=s}$ denotes the residue at $s$. We are interested in the asymptotics of two resonance counting functions. The first one (for the second one see~\eqref{countfct2} below) counts the number of resonances in growing balls centered at the origin $0\in\CC$:
\[
 N_X(r) \sceq \#\{ s\in \mc R(X) \colon |s|\leq r\}, \qquad r>0.
\]
Classically, one would ask for the number of resonances in balls centered at $1/2$:
\[
 \wt N_X(r)\sceq \#\left\{ s\in \mc R(X) \colon \left|s-\tfrac12\right|\leq r\right\}, \qquad r>0.
\]
However, since 
\[
 \wt N_X(r) \leq N_X\left(r+\tfrac12\right) \leq \wt N_X(r+1),
\]
all counting results considered in this article are identical for $N_X$ and $\wt N_X$ (up to the values of some implied or unspecified constants). It is slightly more convenient to work with $N_X$.

If $X$ is a \textit{compact} hyperbolic surface then all resonances arise from $L^2$-eigenvalues of $\Delta_X$, and the Weyl law states the asymptotics
\[
 \frac12 N_X(r) \sim \#\left\{\lambda<r^2 : \text{$\lambda$ is $L^2$-eigenvalue} \right\} \sim \frac{\vol(X)}{4\pi}r^2 \qquad\text{as $r\to\infty$.}
\]
For hyperbolic spaces $X$ of \textit{finite area}, resonances arise not only from $L^2$-eigenvalues but also as scattering poles. By taking into account the contribution of the scattering poles, Selberg~\cite{Selberg_Goe} could establish an analogue of the Weyl law for these spaces:
\begin{equation}\label{Weyl_finite}
 \#\left\{\lambda<r^2 : \text{$\lambda$ is $L^2$-eigenvalue} \right\} - \frac1{4\pi}\int_{-r}^r \frac{\phi'}{\phi}\left(\tfrac12+it\right)\,dt \sim \frac{\vol(X)}{4\pi}r^2 
\end{equation}
as $r\to\infty$, where $\phi$ denotes the determinant of the scattering matrix of $X$. W.~M\"uller~\cite{Mueller_scattering} proved that \eqref{Weyl_finite} yields a \textit{Weyl law for the resonance set}:
\begin{equation}\label{Weyl_resonances}
 N_X(r) \sim \frac{\vol(X)}{2\pi}r^2.
\end{equation}
For hyperbolic surfaces $X$ of \textit{infinite area}, such a Weyl law for the resonance set is not known yet, and probably not even to be expected. At the current status of art we cannot rule out any of the possibilities. 

For the elementary hyperbolic surfaces, i.\,e., for the hyperbolic plane $\hh$, the hyperbolic cylinders \[C_\ell\sceq \langle z\mapsto e^\ell z\rangle\backslash\hh,\] and the parabolic cyclinders \[C_w \sceq \langle z\mapsto z+w\rangle\backslash\hh,\] the resonance sets are precisely known. Their resonance counting functions satisfy
\[
 N_{\hh}(r) \sim r^2,\qquad N_{C_\ell}(r)\sim \frac{\ell}{2} r^2,\qquad N_{C_w}(r) = 1.
\]

For general geometrically finite, non-elementary hyperbolic surfaces $X$ of infinite area, 
Guillop\'e and Zworski \cite{GZ_upper_bounds,GZ_scattering_asympt} showed that the order of growth of the resonance counting function $N_X$ is as for hyperbolic surfaces of finite area, thus
\begin{equation}\label{ordergrowth}
 N_X(r) \asymp r^2
\end{equation}
(which subsumes the elementary hyperbolic surfaces as well by noting that the lower bound as provided by \cite{GZ_scattering_asympt} is allowed to be zero if the $0$-volume vanishes). The implied constants in \eqref{ordergrowth} necessarily depend on $X$ but, unfortunately, a deeper understanding of the geometric content of these constants does not seem to follow from the proofs in \cite{GZ_upper_bounds,GZ_scattering_asympt}.

Borthwick \cite{sharpbounds} established the bounds (for non-elementary hyperbolic surfaces of finite area as well as of infinite area)
\begin{align}\label{Borthwick1}
N_X(r) \leq \left( \frac{\zvol(X)}{2\pi} + \sum_{j=1}^{n_f} \frac{\ell_j}{4}\right) \exp(1) r^2 + o(r^2)
\intertext{and} 
c_k \frac{\zvol(X)}{2\pi}\left( 1 + \frac{2\pi}{\zvol(X)} \sum_{j=1}^{n_f} \frac{\ell_j}{4}\right)^{-\frac{2}{k}} r^2 \leq N_X(r), \label{Borthwick2}
\end{align}
where $\zvol(X)$ denotes the $0$-volume of $X$, $n_f$ is the number of funnels, $\ell_1,\ldots, \ell_{n_f}$ are the lengths of the periodic funnel geodesics (i.e., those periodic geodesics bounding of the funnels), $k$ is any element of $\N$, and $c_k$ is a (rather unspecific) constant depending on $k$ only (in particular, $c_k$ is independent of $X$). The limit for the $o$-term is $r\to\infty$, its speed of convergence may depend on $X$. Obviously, these bounds have a clear geometric content. For $n_f=0$ (i.\,e., for hyperbolic surfaces of finite area) they coincide with \eqref{Weyl_resonances}, and for the hyperbolic cylinder $N_{C_\ell}$ they are sharp. At the current state of art, it is not known if these bounds are sharp in general.

However, more is known if we ask for the transformation behavior of the bounding constants under transitions to covers. At first we note that if $Y=\Gamma\backslash\hh$, $\wt Y = \wt\Gamma\backslash\hh$ are hyperbolic surfaces of finite area with $\wt\Gamma\subseteq \Gamma$ then the constant in the Weyl law~\eqref{Weyl_resonances} scales by $[\Gamma:\wt\Gamma]$ when passing from the asymptotics of $N_{Y}$ to those of $N_{\wt Y}$:
\[
\frac{\vol(\wt Y)}{2\pi}r^2 = [\Gamma:\wt\Gamma] \frac{\vol(Y)}{2\pi}r^2. 
\]
If we now let $X=\Gamma\backslash\hh$, $\wt X = \wt\Gamma\backslash\hh$ be any non-elementary, geometrically finite hyperbolic surfaces with $\wt\Gamma\subseteq\Gamma$ then the $0$-volume has the same scaling behavior, thus
\[
 [\Gamma:\wt\Gamma]\cdot\zvol(X) = \zvol(\wt X).
\]
A straightforward geometric argument shows that also the sum with the funnel lengths in~\eqref{Borthwick1} transforms accordingly, namely
\[
 \sum_{j=1}^{n_f(\wt X)} \frac{\ell_j(\wt X)}{4} = [\Gamma : \wt\Gamma] \cdot \sum_{j=1}^{n_f(X)} \frac{\ell_j(X)}{4}.
\]
Therefore the bounds~\eqref{Borthwick1} and~\eqref{Borthwick2} imply the existence of constants $C_{X,1},C_{X,2}>0$ and $r_X>0$ such that for all $r>r_X$ and all choices of~$\wt X$ we have
\begin{equation}\label{eq:gotocovers}
 C_{X,2} \, [\Gamma : \wt\Gamma] r^2 \leq N_{\wt X}(r) \leq C_{X,1}\, [\Gamma : \wt\Gamma] r^2.
\end{equation}
In order to guarantee that the constant~$C_{X,1}$ is indeed uniform for all covers~$\wt X$ of~$X$, we need to join the $o$-term in~\eqref{Borthwick1} and the first bounding term in~\eqref{Borthwick1}, thus, enlarging $\exp(1)$ by an additive constant depending on~$X$. At the current state of art we are therefore not able to provide a full interpretation of the geometric content of~$C_{X,1}$. Nevertheless, the bounds~\eqref{Borthwick1} and~\eqref{Borthwick2} provided by Borthwick immediately imply (even though not stated in~\cite{sharpbounds}) that the bounding constants for the resonances counting function $N_X$ can be chosen such that they scale as in the Weyl law~\eqref{Weyl_resonances} when passing to covers (see also the discussion following Theorem~\ref{thm:counting}). Borthwick's proof is based on microlocal analysis and spectral theory.

A related result, obtained by means of thermodynamic formalism and transfer operator techniques, has been obtained by Jakobson and Naud~\cite{JN}. They considered convex cocompact hyperbolic surfaces $X$ (no singularities!) of infinite area, thus, geometrically finite hyperbolic surfaces without cusps and with at least one funnel. In this article we refer to such surfaces as \textit{Schottky surfaces}. 

They restricted further to those Schottky surfaces $X$ for which the fundamental group $\Gamma$ (identified with a subgroup of $\PSL_2(\RR)$, see Section~\ref{sec:prelims}) is (conjugate to a subgroup) in $\PSL_2(\Z)$. For the purpose of this article, we call such Schottky groups \textit{integral}. Given such an integral Schottky surface $X=\Gamma\backslash\hh$, Jakobson and Naud considered the sequence of finite covers 
\[
X_q=\Gamma(q)\backslash \hh, \qquad \text{$q\in\N$ prime}
\] 
where
\[
 \Gamma(q) \sceq \left\{ g\in \Gamma : g \equiv \id \mod q\right\}
\]
is the `principal congruence subgroup' of $\Gamma$ of level $q$. They showed that
\begin{itemize}
\item there exist constants $C_1>0$, $q_0\in\N$ (possibly depending on $X$) such that for all $q\geq q_0$, $q$ prime, and all $r\geq 1$ we have
\begin{equation}\label{JN_lower}
 N_{X_q}(r) \leq C_1 [\Gamma : \Gamma(q)] \log(q) r^2,
\end{equation}
\item and there exist constants $C_2, r_0>0$ (possibly depending on $X$) such that for all $\eps>0$ there exists $q_0\in\N$ such that for all $q\geq q_0$, $q$ prime, and all $r>r_0$ we have
\begin{equation}\label{JN_upper}
 N_{X_q}\big( r\cdot (\log q)^\eps \big) \geq C_2 [\Gamma : \Gamma(q)] r^2.
\end{equation}
\end{itemize}

The estimates \eqref{JN_upper} and \eqref{JN_lower} show that transfer operator techniques allow us to prove that for the transition from an integral Schottky surface $X$ to any of its principal congruence covers, the upper and lower bounds for the resonance counting function of $X$ can be chosen in a way that they transform \emph{almost} as the constant in the Weyl law \eqref{Weyl_resonances}. (This is a weaker result than Borthwick's; the important point here is the different methodology.)

Our first main result is a generalization and improvement of these transfer operator techniques. In order to simplify the statement of the result we let, for any Schottky surface $X$ and any cover $\wt X$ of $X$, 
\[
 \dcov(\wt X,X)
\]
denote the degree of $\wt X$ as a cover of $X$.

\begin{thm}\label{thm:counting}
Let $X$ be a Schottky surface (not necessarily integral).
\begin{enumerate}[{\rm (i)}]
\item\label{counti} There exists a constant $C_1>0$ such that for each finite cover $\wt X$ of $X$ and all $r\geq 1$ we have
\[
 N_{\wt X}(r) \leq C_1 \dcov(\wt X,X) r^2.
\]
\item\label{countii} There exist constants $C_2, r_0>0$ such that for each finite cover $\wt X$ of $X$ and all $r\geq r_0$ we have
\[
 N_{\wt X}(r) \geq C_2 \dcov(\wt X, X) r^2.
\]
\end{enumerate}
\end{thm}

We list a few remarks about Theorem~\ref{thm:counting} and its relation to the counting results mentioned above.

\begin{itemize}
\item If $X=\Gamma\backslash\hh$ and $\wt X = \wt\Gamma\backslash\hh$ then 
\[
 \dcov(\wt X,X) = [\Gamma:\wt\Gamma].
\]
Thus, Theorem~\ref{thm:counting} recovers the relation in~\eqref{eq:gotocovers} for Schottky surfaces, and it shows again that the bounding constants for the resonance counting function can be chosen such that they scale exactly as the constants in the Weyl law \eqref{Weyl_resonances} when passing to covers. In other words, transfer operator techniques and techniques from microlocal analysis produce results of same strength.
\item The hyperbolic cylinders $C_\ell$ are Schottky surfaces. For these it is easy to see that the bounds in Theorem~\ref{thm:counting} are sharp with $C_1=C_2=\ell/2$, see also the more detailed discussion in Section~\ref{sec:proof1} below.
\item Theorem~\ref{thm:counting} obviously applies to sequences of principal congruence covers of integral Schottky surfaces, and it improves upon the result in~\cite{JN}.  
\item As already mentioned above, for any finite cover $\wt X$ of a Schottky surface~$X$ we have the relation
\[
 \dcov(\wt X,X)\cdot\zvol(X) = \zvol(\wt X).
\]
For non-elementary Schottky surfaces (in which case $\zvol(X)\not=0$) this relation can be read as
\begin{equation}\label{dcovzvol}
 \dcov(\wt X,X) = \frac{\zvol(\wt X)}{\zvol(X)}.
\end{equation}
Using \eqref{dcovzvol} in Theorem~\ref{thm:counting} and merging the term $\zvol(X)$ into the constants $C_1,C_2$ (which are allowed to depend on $X$) gives
\begin{equation}\label{fancy}
 C_2 \zvol(\wt X)r^2 \leq N_{\wt X}(r) \leq C_1 \zvol(\wt X) r^2,
\end{equation}
which is reminiscent of \eqref{Weyl_resonances} and justifies to understand the bounds of Theorem~\ref{thm:counting} as upper and lower Weyl-type bounds. Consequently, Theorem~\ref{thm:counting} shows that for any sequence of finite covers $ (X_{n})_{n} $ of a non-elementary Schottky surface $ X $ we have a weak Weyl law
$$ N_{X_{n}}(r) \asymp \zvol(X_{n}) r^2, $$ 
with implied constants only depending on the base surface $ X $. As for the techniques used to establish~\eqref{eq:gotocovers},  we do not understand yet if (and how) transfer operator techniques might allow us to provide any further insight about the unspecified constants $C_1,C_2$ in Theorem~\ref{thm:counting}. 

\item Theorem~\ref{thm:counting} is stated with $\dcov(\wt X, X)$ instead of using (the arguably more intriguing variants) \eqref{dcovzvol} and \eqref{fancy} in order to be able to subsume the hyperbolic cylinders into the statement (note that $\zvol(C_\ell) = 0$).
\end{itemize}

The second resonance counting function we investigate in this article is
\begin{equation}\label{countfct2}
 M_X(\sigma, T) \sceq \#\left\{ s\in \mc R(X) : \Rea s \geq \sigma,\ |\Ima s - T | \leq 1\right\}
\end{equation}
for $\sigma, T\in\RR$. For any hyperbolic surface $X = \Gamma\backslash\hh$ it is known that the right half plane $\{ s\in\CC \colon \Rea s > \delta\}$ does not contain any resonances, where 
\[
\delta \sceq \delta(X)\sceq \dim \Lambda(\Gamma)
\]
is the Hausdorff dimension of the limit set of $\Gamma$ or, equivalently, of $X$. Thus, $M_X(\sigma, T)$ is counting the number of resonances in the box 
\[
 [\sigma, \delta] + i[T-1,T+1].
\]
The counting function $M_X(\sigma, T)$ is closely related to 
\[
 N_X(\sigma, T) \sceq \#\left\{ s\in \mc R(X) : \Rea s \geq \sigma,\ 0\leq\Ima s \leq T \right\},
\]
which asks for the number of resonances in a strip near $\delta$ as $T\to\infty$. The fractal Weyl law conjecture \cite{Sjoestrand, Lu_Sridhar_Zworski} for hyperbolic surfaces predicts the asymptotics 
\[
  N_X(\sigma, T) \asymp T^{1+\delta}.
\]
For hyperbolic surfaces of finite area (in this case $\delta=1$), the fractal Weyl law conjecture follows from \eqref{Weyl_resonances}. For hyperbolic surfaces of infinite area, it is still under investigation.

Clearly, every asymptotics for $M_X(\sigma, T)$ yields one for $N_X(\sigma, T)$. For Schottky surfaces $X$, Guillop\'e--Lin--Zworski \cite{Guillope_Lin_Zworski} showed that for any $\sigma\in\RR$ we have, as $ T\to \infty $, an upper fractal Weyl bound 
\begin{equation}\label{fractalWeylupper}
 M_X(\sigma, T) = O_{\sigma}(T^{\delta})
\end{equation}
Dyatlov \cite{BDW} recently improved this estimate to give
\begin{equation}\label{fractalWeylupper_improved}
 M_X(\sigma, T) = O_{\sigma, \varepsilon}( T^{\min \{ 2(\delta-\sigma), \delta \}+\varepsilon} ), 
\end{equation}
for all $ \varepsilon > 0 $, showing that the exponent in \eqref{fractalWeylupper} can be improved if $\sigma\in (\delta/2,\delta]$. A recently achieved upper fractal Weyl bound for (a class of) hyperbolic surfaces of infinite area with cusps is provided in~\cite{NPS}. The nature of the relation between the geometry of~$X$ and the implied constants in~\eqref{fractalWeylupper} and~\eqref{fractalWeylupper_improved} is not investigated in~\cite{Guillope_Lin_Zworski, BDW, NPS}. 

In \cite{JN} (prior to \cite{BDW}, and with different techniques), Jakobson and Naud studied the transformation behavior of the bounding constant under the transition from an integral Schottky surface to any of its principal congruence covers. They found functions $\alpha, \beta\colon\RR\to\RR$ that are strictly concave, increasing, and positive on $(\delta/2,\delta]$ such that for each $\sigma>\delta/2$ there exists $C>0$ and $q_0\in\N$ such that for all $T\geq 1$ and all levels $q\geq q_0$, $q$ prime, we have
\begin{equation}\label{JNfractal}
 M_{X_q}(\sigma, T) \leq C [\Gamma : \Gamma(q)]^{1-\alpha(\sigma)} \langle T\rangle^{\delta-\beta(\sigma)}.
\end{equation}
For $\sigma\in(\delta/2,\delta]$, the bound \eqref{JNfractal} simultaneously improves upon the upper fractal Weyl bound~\eqref{fractalWeylupper} and shows how the bounding constants behave in the level aspect. 

Our second main result is a generalization of \eqref{JNfractal} to arbitrary Schottky surfaces and arbitrary finite covers.

\begin{thm}\label{thm:box}
Let $X$ be a Schottky surface, and let $\delta\sceq\delta(X)$ denote the Hausdorff dimension of its limit set. Then there exist functions $\tau_1,\tau_2\colon\RR\to\RR$ that are strictly concave, strictly increasing and positive on $(\delta/2,\delta]$ such that for every $\sigma > \delta/2$ there exists $C>0$ such that for each finite cover $\wt X$ of $X$ and all $T\in\RR$ we have
\begin{equation}\label{newfraccover}
 M_{\wt X}(\sigma, T)\leq C \zvol(\wt X) e^{-\tau_1(\sigma) \ell_0(\wt X)} \langle T \rangle^{\delta-\tau_2(\sigma)},
\end{equation}
where $\langle T\rangle \sceq \sqrt{1+|T|^2}$ and $\ell_0(\wt X)$ denotes the minimal length of a periodic geodesic on $\wt X$.
\end{thm}

For hyperbolic cyclinders, Theorem~\ref{thm:box} is obviously valid since both sides vanish (see the more detailed discussion on the location of resonances in Section~\ref{sec:prelims} below). The functions $\tau_1$ and $\tau_2$ can be determined rather explicitly. We refer to Section~\ref{sec:proof2} below for a few more details.

Integrating the bound along $T$ yields the following counting statement, which should be seen as an extension of \cite[Theorem~1.1]{Naud_inventiones} which includes the transformation behavior of the bounding constants under the transition to covers.

\begin{cor}\label{cor:strip}
With hypotheses and notation as in Theorem~\ref{thm:box} we have
\[
 N_{\wt X}(\sigma, T)\leq C \zvol(\wt X) e^{-\tau_1(\sigma) \ell_0(\wt X)} \langle T \rangle^{1+\delta-\tau_2(\sigma)}.
\] 
\end{cor}

Without the aspect of the transition to covers, Corollary~\ref{cor:strip} is \cite[Theorem~1.1]{Naud_inventiones}, albeit as stated slightly weaker in respect to $\tau_2$. Naud \cite{Naud_inventiones} showed further properties of the function $\tau_2$. His construction of $\tau_2$ applies here as well (and is indeed used in the proof of Theorem~\ref{thm:box}), and hence a more careful analysis will show analogous properties for this function in our situation. As \cite{BDW}, \cite[Theorem~1.1]{Naud_inventiones} (which is older than \cite{BDW} and uses different techniques) shows that the upper fractal Weyl bound \eqref{fractalWeylupper} can be improved near $\delta$. Corollary~\ref{cor:strip} shows that the improvement is uniform along covers.

If $\wt X = \wt\Gamma\backslash\hh$ is a finite \textit{regular} cover of $X=\Gamma\backslash\hh$, that is, $\wt\Gamma$ is normal in $\Gamma$, then Theorem~\ref{thm:box} gives rise to an upper bound of $M_{\wt X}$ in terms of the girth of the Cayley graph of $\Gamma/\wt\Gamma$. We discuss this further in Section~\ref{sec:cayley} below.

As observed by Jakobson--Naud \cite{JN}, along sequences of principal congruence covers $(X_q)_q$ the bound \eqref{JNfractal} implies the growth estimate
\begin{equation}\label{JNgrowth}
 \#\{ \text{$\lambda$ Laplace $L^2$-eigenvalue of $X_q$}\} = O(\zvol(X_q)^{1-\eps}) \quad\text{as $q\to\infty$}
\end{equation}
for some $\eps>0$. A similar estimate (in more generality) was recently shown by Oh \cite{Oh}, and the same conclusion can be deduced from \eqref{newfraccover}, see Proposition~\ref{prop:eigengrowth} below. These estimates complement the recent bounds by Ballmann, Matthiesen and Mondal \cite{BMM}.

To be more precise, let $X$ be a non-elementary Schottky surface, and let $\Omega(X)$ denote its (multi-)set of Laplace $L^2$-eigenvalues. By applying Theorem~\ref{thm:box} to $T=0$ we find constants $C,\tau_1>0$ such that for every finite cover $\wt X$ of $X$ we have 
\[
 \frac{\#\Omega(\wt X)}{\zvol(\wt X)} \leq C e^{-\tau_1\ell_0(\wt X)}.
\]
If $(X_n)_n$ is a sequence of finite covers of $X$ such that $\ell_0(X_n) \to \infty$ as $n\to\infty$ then 
\begin{equation}\label{limitseq}
 \frac{\#\Omega(X_n)}{\zvol(X_n)} \to 0\quad \text{as $n\to\infty$.}
\end{equation}
If $\delta(X)<1/2$ then \eqref{limitseq} is trivial because $\#\Omega(X_n) = 0$ in this case. However, for $\delta(X)>1/2$, Laplace eigenvalues are known to exist. By \cite{BMM},
\[
 \#\Omega(X) \leq -\chi(X),
\]
where $\chi(X)$ denotes the Euler characteristic of $X$. (The result by Ballmann, Matthiesen and Mondal is in fact much more general. It applies to all geometrically finite hyperbolic surfaces.) 

Since $\zvol(X) = -2\pi\chi(X)$, along any sequence $(\wt X_n)_n$ satisfying \eqref{limitseq}, $\#\Omega(X_n)$ grows slower than $-\chi(X_n)$ as $n\to\infty$. Sequences of principal congruence covers of integral Schottky surfaces provide such examples. In Section~\ref{sec:eigen} below we provide further classes of examples. 

We provide a brief overview of the structure of the article. The proofs of Theorem~\ref{thm:counting} and \ref{thm:box} are based on thermodynamic formalism and transfer operator techniques. In particular, we make use of the standard transfer operator $\TO_s$ for Schottky surfaces $X=\Gamma\backslash\hh$ and its variants $\TO_{s,\varrho}$ that are twisted with finite-dimensional unitary representations $\varrho\colon \Gamma\to U(V)$. The Fredholm determinant of $\TO_{s,\varrho}$ is known to be equal to the $L$-function (twisted Selberg zeta function)
\[
 L_\Gamma(s,\varrho) = \prod_{\overline{\gamma}\in \overline{\Gamma}_{p}} \prod_{k=0}^{\infty} \det\left(  1-\varrho(\gamma)e^{-(s+k)\ell(\gamma)} \right), \qquad \Rea s \gg 1
\]
and its analytic continuation to all of $\CC$ (see Section~\ref{sec:prelims} below for notation). Thus,
\begin{equation}\label{fredholm}
 L_\Gamma(s,\varrho) = \det\left( 1 -\TO_{s,\varrho}\right).
\end{equation}
The specific structure of these transfer operators for Schottky surfaces allows us to separate the contribution of the representation $\varrho$ in the transfer operator $\TO_{s,\varrho}$ from the dynamical parts (see Section~\ref{sec:proof1} below). Combining this separation with \eqref{fredholm} and the known growth estimates of the singular values of $\TO_s$ enable us to establish the following result on the growth of $L_\Gamma$, which is a key ingredient for the proof of Theorem~\ref{thm:counting}.

\begin{prop}\label{prop:twist_growth}
Let $\Gamma$ be a Schottky group. Then there exists $C>0$ such that for every finite-dimensional unitary representation $\varrho$ of $\Gamma$ and all $s\in\CC$ we have
\[
 \log \big|L_\Gamma(s,\varrho)\big| \leq C \cdot \dim\varrho \cdot \langle s\rangle^2.
\]
\end{prop}

In Section~\ref{sec:prelims} below we introduce the necessary background knowledge on Schottky surfaces and the transfer operators. Sections~\ref{sec:proofprop} and \ref{sec:proof1} are devoted to the proofs of Proposition~\ref{prop:twist_growth} and Theorem~\ref{thm:counting}, respectively. In Section~\ref{sec:proof2} we provide a proof of Theorem~\ref{thm:box}. The final two Sections~\ref{sec:eigen} and \ref{sec:cayley} discuss examples for \eqref{limitseq} and a relation of Theorem~\ref{thm:box} to Cayley graphs, respectively.

\subsubsection*{Acknowledgement} The authors are grateful to Fr\'ed\'eric Naud for helpful discussions on various aspects of this work. Further, AP acknowledges support by the DFG grant PO 1483/2-1.

\section{Preliminaries and Notation}\label{sec:prelims}

\subsection{Hyperbolic surfaces} 
Throughout we use the upper half plane model of the hyperbolic plane
$$ \hh = \{ z = x+iy : x\in \mathbb{R}, y >0\}, \quad
ds^{2} = \frac{dx^{2}+dy^{2}}{y^{2}}.
$$
In these coordinates, the associated (positive) Laplace--Beltrami operator is 
\[
 \Delta_\hh = -y^2\big(\partial_x^2 + \partial_y^2\big).
\]
The group of orientation-preserving isometries of $ \hh $ is isomorphic to the group $ \mathrm{PSL}_{2}(\mathbb{R}) = \mathrm{SL}_{2}(\mathbb{R})/ \{ \pm \id \} $, acting by M\"{o}bius transformations on $ \hh $. This action extends continuously to the geodesic boundary $ \partial \hh $ of $ \hh $, which we identify with $ \overline{\mathbb{R}}:= \mathbb{R}\cup \{ \infty \} $. The action of $ \mathrm{PSL}_{2}(\mathbb{R}) $ on the geodesic closure $ \overline{\hh} = \hh \cup \partial \hh $ is then given by
$$
g. z := \begin{cases}  \infty &\text{ if } z=\infty,\; c = 0, \text{ or } z\neq \infty,\, cz+d=0, \\[+2mm] \dfrac{a}{c} &\text{ if } z=\infty,\; c \neq  0,\\[+2mm] \dfrac{az+b}{cz+d} &\text{ otherwise}   \end{cases}
$$ 
for $ g = \textbmat{a}{b}{c}{d} \in \mathrm{PSL}_{2}(\mathbb{R}) $ and $ z\in \overline{\hh} $.

Every hyperbolic surface $X$ is isometric to a quotient $ \Gamma\backslash\hh $ for some Fuchsian group $ \Gamma$, that is, a discrete subgroup of $\PSL_{2}(\mathbb{R})$. 

We recall that an element $ g\in \mathrm{PSL}_{2}(\mathbb{R}) $, $g\not=\id$, is called \textit{hyperbolic} if it has precisely two fixed points $ z_{1}, z_{2} $ on $ \partial \hh $ (equivalently, $|\Tr g|>2$).  It is called \textit{elliptic} if it has a single fixed point in $\hh$ (equivalently, $|\Tr g|<2$), and it is called \textit{parabolic} if it has a single fixed point in $\partial\hh$ (equivalently, $|\Tr g|=2$). 

Let $\Gamma$ be a Fuchsian group. An element $g\in\Gamma$, $g\not=\id$, is called \textit{primitive} if for every $h\in\Gamma$ and $n\in\N$, $g=h^n$ implies $n=1$. We let $\overline{\Gamma}_p$ denote the set of $\Gamma$-conjugacy classes of the primitive hyperbolic elements in $\Gamma$, and $\overline\Gamma$ the set of $\Gamma$-conjugacy classes of all hyperbolic elements in $\Gamma$. Throughout we take advantage of the well-known bijection between $\overline{\Gamma}_p$ and the set of primitive periodic geodesics on $ X = \Gamma\setminus \hh $ as well as the bijection between $\overline{\Gamma}$ and the set of all periodic geodesics on $X$ (allowing multiple passages through the image).  If $\gamma$ is a periodic geodesic on $X$ and $[g]$ the corresponding class in $\overline\Gamma$ then the \textit{length} of $\gamma$ is given by 
\[
 \ell(\gamma) = \ell(g) = \log N_G(g),
\]
where 
\begin{equation}\label{defN}
 N_G(g) \sceq \max\{ |\lambda|^2 : \text{$\lambda$ eigenvalue of $g$} \}
\end{equation}
denotes the norm of $g$. We denote the length of the shortest geodesic on $X$ by $\ell_0(X)$, hence
\begin{equation}\label{mingeod}
\ell_0(X) = \min_{\overline g\in\overline\Gamma} \ell(g).
\end{equation}

\subsection{Schottky groups and Schottky surfaces}\label{sec:Schottky}

Throughout we restrict all considerations to Schottky surfaces and (Fuchsian) Schottky groups.  (Fuchsian) \textit{Schottky groups} are precisely those Fuchsian groups that are geometrically finite, not cofinite and have no elliptic and parabolic elements. We call a hyperbolic surface $X$ \textit{Schottky} if there exists a Schottky group $\Gamma$ such that $X=\Gamma\backslash\hh$. The class of Schottky surfaces coincides with the class of convex cocompact hyperbolic surfaces of infinite area (and without singularities).

Classically, Schottky groups are given by the following geometric construction of which we take advantage in this article as well:

Let $m\in\N$ and choose $2m$ open Euclidean disks in $\CC$ that are centered on $\partial\hh $ and have mutually disjoint closures.  Endow these disks with an ordering, say
\begin{equation}\label{disks}
  \mathcal{D}_{1}, \dots, \mathcal{D}_{2m}.
\end{equation}
Note that the action of $\PSL_2(\RR)$ on $\overline{\hh}$ by M\"obius transformations extends continuously to the whole Riemann sphere $\overline{\CC}$. For $j\in \{ 1, \dots, m\}$ let $\gamma_{j}$ be an element in $\PSL_2(\RR)$ that maps the exterior of $ \mathcal{D}_{j} $ to the interior of $ \mathcal{D}_{j+m}$. Then the elements $ \gamma_{1}, \dots, \gamma_{m} $ and its inverses freely generate a Schottky group (as a group). 

Conversely, every Schottky group $\Gamma$ is conjugate within $\PSL_2(\RR)$ to a Schottky group arising from this construction for some $m\in\N$, a family of open disjoint disks $(\mathcal D_j)_{j=1}^{2m}$ and a family of elements $(\gamma_j)_{j=1}^m$ in $\PSL_2(\RR)$ with the mapping properties as above, with all objects depending on $\Gamma$ (more precisely, on the choice of conjugate) and are subject to choices. To be more precise, if one uses this construction with disks in $\overline{\CC}$ (as it is classically done) then every Schottky group is given by this construction. However, for our applications, we will use only Schottky groups for which all disks in \eqref{disks} are indeed contained in $\CC$. This restriction does not restrict our results, since we are mostly interested in the investigation of Schottky surfaces $X=\Gamma\backslash\hh$ (in which case we are free to choose $\Gamma$ such that no such conjugation is needed), and the remaining results are invariant under  conjugations.

For convenience, given a Schottky group $\Gamma$, we omit throughout any reference to a possibly necessary conjugation. We use $m\in\N$, fix a choice of the disks in \eqref{disks} and the generators $\gamma_1,\ldots, \gamma_m$ from a (fixed) geometric construction of $\Gamma$ without any further reference. Moreover, for $j\in \{m+1,\ldots, 2m\}$, we set 
\[
 \gamma_j \sceq \gamma_{j-m}^{-1},
\]
and extend this definition cyclically to $\Z$ by definining 
\[
 \gamma_k \sceq \gamma_{k\text{ mod }2m}\quad \text{for $k\in\Z$.}
\]
Further, we let
\begin{equation}\label{def_D}
\mc D\sceq \bigcup_{j=1}^{2m} \mc D_j
\end{equation}
be the union of the disks \eqref{disks} used in the construction.

\subsection{Resonances}

We recall the definition and properties of resonances of the Laplacian for Schottky surfaces only. Let $X=\Gamma\backslash\hh$ be a Schottky surface, let $\Delta_X$ denote the Laplacian on $X$, and let 
\[
 \delta \sceq \delta(X) \sceq \dim\Lambda(\Gamma)
\]
denote the Hausdorff dimension of the limit set $\Lambda(\Gamma)$ of $\Gamma$. 

The spectrum of $\Delta_X$ on $L^2(X)$ is rather sparse (much in contrast to the resonance set, defined further below). By Lax-Phillips \cite{Lax_Phillips_I} and Patterson \cite{Patterson}, it satisfies the following properties:
\begin{itemize}
\item The (absolutely) continuous spectrum is $ [1/4, \infty) $.
\item The pure point spectrum is finite and contained in $(0,1/4)$. In particular, there are no eigenvalues embedded in the continuous spectrum.
\item If $\delta<1/2$ then the pure point spectrum is empty. If $\delta>1/2$ then $\delta(1-\delta)$ is the smallest eigenvalue.
\end{itemize}
The resolvent 
\[
 R_X(s) \sceq \big( \Delta_X - s(1-s)\big)^{-1}\colon L^2(X) \to H^2(X)
\]
of $\Delta_X$ is defined for $s\in\CC$ with $\Rea s>1/2$ and $s(1-s)$ not being an $L^2$-eigenvalue of $\Delta_X$. By \cite{Mazzeo_Melrose, GZ_upper_bounds}, it extends to a meromorphic family 
\[
 R_X(s) \colon L^2_{\text{comp}}(X) \to H^2_{\text{loc}}(X)
\]
on $\CC$ with poles of finite rank. The \textit{resonances} of $X$ are the poles of $R_X$. We denote the set of resonances, repeated according to multiplicities, by 
\[
 \mc R(X).
\]
The set $\mc R(X)$ of resonances is contained in the half space
\[
 \{ s \in\CC : \Rea s \leq \delta\},
\]
and, obviously, each $L^2$-eigenvalue gives rise to a (pair of) resonance(s). We refer to the Introduction for some known results on the finer structure of the set $\mc R(X)$, and recall that we are interested in the study of the two resonance counting functions
\[
 N_X(r) \sceq \#\{ s\in \mc R(X) \colon |s|\leq r\},\quad r>0,
\]
and
\[
 M_X(\sigma, T) \sceq \#\left\{ s\in \mc R(X) : \Rea s \geq \sigma,\ |\Ima s - T | \leq 1\right\},\quad \sigma, T\in\RR.
\]

\subsection{Representation} 

For any Fuchsian group $\Gamma$ and any finite-dimensional representation $\varrho\colon \Gamma\to U(V)$ of $\Gamma$ on a finite-dimensional unitary space $V$ we denote throughout the inner product on $V$ by $\langle\cdot,\cdot\rangle$ and its associated norm by $\|\cdot\|$, without any further reference to $V$. We define the \textit{dimension} of $\varrho$ to be the dimension of $V$:
\[
 \dim\varrho \sceq \dim V.
\]
Further, we denote by $\trivrep_V$ the trivial representation of $\Gamma$ on $V$ if $\Gamma$ is understood implicitly, and by $\trivrep_{\Gamma}$ the trivial representation of $\Gamma$ on $V$ if $V$ is understood implicitly.

\subsection{Selberg zeta functions and $L$-functions}\label{sec:SZF}

Let $X=\Gamma\backslash\hh$ be a Schottky surface and $\varrho\colon \Gamma\to U(V)$ a unitary representation of $\Gamma$ on a finite-dimensional unitary space $V$. 

The $L$-function (twisted Selberg zeta function) associated to $(\Gamma,\varrho)$ is (initially only formally) determined by the Euler product (recall $N_G(\gamma)$ from \eqref{defN})
\begin{align}\label{Euler_product}
L_{\Gamma}(s, \varrho) & = \prod_{\overline{\gamma}\in \overline{\Gamma}_{p}} \prod_{k=0}^{\infty} \det\left(  1-\varrho(\gamma)e^{-(s+k)\ell(\gamma)} \right)
\\
& = \prod_{\overline{\gamma}\in \overline{\Gamma}_{p}} \prod_{k=0}^{\infty} \det\left(  1-\varrho(\gamma)N_G(\gamma)^{-(s+k)} \right) \nonumber
\end{align}
The infinite product \eqref{Euler_product} converges compactly on $\{ s\in\CC : \Rea s > \delta(X)\}$, and it has an analytic continuation to all of $\CC$ (see, e.\,g., \cite{FP_szf, JNS} and the additional comments in Section~\ref{sec:TO} below).

For the one-dimensional trivial representation $\trivrep_{\Gamma}$, the $L$-function is the classical Selberg zeta function $Z_{\Gamma}$ of $\Gamma$ (or of $X$), and \eqref{Euler_product} is the Euler product of $Z_\Gamma$:
\[
 Z_\Gamma(s) = L_{\Gamma}(s, \trivrep_{\Gamma}) = \prod_{\overline{\gamma}\in \overline{\Gamma}_{p}} \prod_{k=0}^{\infty} \left(  1-e^{-(s+k)\ell(\gamma)} \right).
\]

By~\cite{Patterson_Perry} (see also~\cite{BJP}) the set $\mc R(X)$ of resonances of $X$ is contained in the set of zeros of $Z_\Gamma$, counted with multiplicities. This property of the Selberg zeta function allows us to translate upper estimates of the resonance counting functions $N_X$ and $M_X$ into counting problems of the number of zeros of $Z_\Gamma$ in certain domains.

\subsection{Transfer operators}\label{sec:TO}

The thermodynamic formalism for Selberg-type zeta functions allows to represent the Selberg zeta function and the $L$-functions considered in this article as Fredholm determinants of well chosen transfer operators. The transfer operators used here derive from a certain discretization of the geodesic flow on the considered Schottky surface. We refer to \cite{Patterson_Perry, Guillope_Lin_Zworski, Borthwick_book} for details regarding the representation of the Selberg zeta function, and to \cite{FP_szf, JNS} for the extension to twisted transfer operators and $L$-functions, and remain here rather brief.

Let $\Gamma$ be a Schottky group, let $(\mc D_j)_{j=1}^{2m}$ be the family of open disks in $\CC$, and $(\gamma_j)_{j=1}^{2m}$ the family of elements in $\PSL_2(\RR)$ used in a (fixed) geometric construction of $\Gamma$ (see Section~\ref{sec:Schottky}), and recall that 
\[
 \Gamma = \langle \gamma_1^{\pm 1},\ldots, \gamma_m^{\pm1}\rangle
\]
is freely presented as a group (thus, the only (omitted) relations are of the form $\gamma\gamma^{-1}=\id$). Set 
\[
\mc D \sceq \bigcup_{j=1}^{2m} \mc D_j.
\]
Let $\varrho\colon\Gamma\to U(V)$ be a finite-dimensional unitary representation of $\Gamma$. The transfer operator $\TO_{s,\varrho}$ with parameter $s\in\CC$ associated to $(\Gamma,\varrho)$ is (initially only formally) given by
\begin{equation}\label{TO_def}
 \TO_{s,\varrho} \sceq \sum_{j=1}^{2m} 1_{\mc D_j} \sum_{\substack{i=1 \\ i\not=j+m}}^{2m} \nu_s(\gamma_i),
\end{equation}
where $1_{\mc D_j}$ denotes the characteristic function of $\mc D_j$, and for $g\in \PSL_2(\RR)$, $U\subseteq\CC$, $f\colon U\to\CC$ we set
\begin{equation}\label{nu}
 \nu_s(g^{-1})f(z) \sceq \big(g'(z)\big)^s \varrho(g^{-1}) f(g.z), 
\end{equation}
whenever this is well-defined. For the complex powers in \eqref{TO_def} we use the standard complex logarithm on $\CC\smallsetminus\RR_{\leq 0}$ (principal arc). A straightforward calculation shows well-definedness on $\mc D$ (which is what is needed in this article).

Depending on the application, the transfer operator $\TO_{s,\varrho}$ is considered to act on different spaces of functions defined on (subsets of) $\mc D$. For the proofs of Proposition~\ref{prop:twist_growth} and Theorem~\ref{thm:counting} we consider $\TO_{s,\varrho}$ to act on the Hilbert space defined in the following. For the proof of Theorem~\ref{thm:box} we use a `refinement' presented in Section~\ref{sec:proof2} below.

In order to define the Hilbert space on which we consider the transfer operator $\TO_{s,\varrho}$ as an actual operator, we let, for each $j\in\{1,\ldots, 2m\}$,  
\[
 \mc H_j \sceq H^2(\mc D_j) \sceq \left\{ \text{$f\colon \mc D_j\to V$ holomorphic} \ \left\vert\ \int_{\mc D_j} \|f\|^2\dvol < \infty \right.\right\}
\]
denote the space of holomorphic square-integrable $V$-(vector-)valued functions on $\mc D_j$. The volume form used in the definition is the standard Lebesgue measure on $\CC$. Endowed with the inner product
\[
 \langle f, g\rangle \sceq \int_{\mc D_j} \langle f(z), g(z)\rangle_{V}\dvol(z),
\]
the space $\mc H_j$ is a Hilbert space, the \textit{Hilbert Bergman space} on $\mc D_j$. Let
\[
 \mc H\sceq \bigoplus_{j=1}^{2m} \mc H_j
\]
denote the direct sum of the Hilbert spaces $\mc H_j$, $j=1,\ldots, 2m$.  As usual, we identify tacitly functions  
\[
f\in\mc H, \quad f=\bigoplus_{j=1}^{2m} f_j \qquad (f_j\in \mc H_j)
\]
with functions on $\mc D$. 

Note that for all $i,j\in\{1,\ldots, 2m\}$, $ i\neq j+m \mod 2m$, we have $\overline{\gamma_{i}^{-1}(\mathcal D_{j})} \subset \mathcal{D}_{i}$. Hence $ \gamma_{i}^{-1} \colon \mathcal{D}_{j}\to \mathcal{D}_{i} $ is a holomorphic contraction, the transfer operator $\TO_{s,\varrho}$ is well-defined as an operator 
\[
 \TO_{s,\varrho}\colon \mc H \to \mc H,
\]
and as such it is compact and of trace class. 

A property crucial for our application is that the Fredholm determinant of the transfer operators $\TO_{s,\varrho}$ represents the $L$-function. For all $s\in\CC$ we have
\begin{equation}\label{FredholIdentity}
L_{\Gamma}(s, \varrho) = \det(1-\mathcal{L}_{s,\varrho}).
\end{equation}
For a proof see, e.\,g., \cite{Guillope_Lin_Zworski}, for the trivial representation, and \cite{FP_szf} (arbitrary geometrically finite hyperbolic surfaces) or \cite{JNS} (specialized on Schottky surfaces) for all representations (which also includes a transfer-operator-based proof of the meromorphic continuability of $L_\Gamma(\cdot,\varrho)$).

\subsection{Some elements of functional analysis}

We recall a few elements of functional analysis that are used throughout this article. For proofs and more details we refer to \cite{Simon} or any other standard reference.

Let $H$ be a separable Hilbert space, and let $A\colon H\to H$ be an operator on $H$ of trace class. We note that some parts of this section apply to operators more general than trace class. However, such generalizations are not needed for our purposes.

Let $A^*$ denote the adjoint of $A$. Then $A^*A$ is positive semi-definite, and hence the \textit{absolute value}  
\[
 |A| \sceq \big(A^*A\big)^\frac12.
\]
of $A$ exists. The \textit{singular values} of $A$ are the non-zero eigenvalues of $|A|$. Let $(\mu_k(A))_{k=1}^{S(A)}$ be the sequence of singular values (with multiplicities) of $A$, arranged by decreasing order:
\[
 \mu_1(A) \geq \mu_2(A) \geq \mu_3(A) \geq \cdots 
\]
If necessary then we turn this sequence into an infinite one by filling it up with zeros at the end. The trace norm of $A$ is
\[
 \|A\|_1 \sceq \sum_{j=1}^\infty \mu_j(A).
\]
Further, let $(\lambda_j(A))_{j=1}^{E(A)}$ be the sequence of eigenvalues (with multiplicities) of $A$, arranged by decreasing absolute value:
\[
 |\lambda_1(A)|\geq |\lambda_2(A)|\geq \cdots.
\]
Then 
\begin{equation}\label{def_det}
 \det(1+A)  = \prod_{j=1}^{E(A)}(1+\lambda_j(A)).
\end{equation}
By the Weyl inequality we have for each $N\in\N$, 
\[
 \prod_{j=1}^N \big(1+|\lambda_j(A)|\big) \leq \prod_{j=1}^N \big(1 + \mu_j(A)\big).
\]
In particular, 
\begin{equation}\label{detsing}
 \big| \det(1+A)\big| = \prod_{j=1}^\infty \big(1 + \lambda_j(A)\big) \leq \prod_{j=1}^\infty \big(1 + \mu_j(A)\big) = \det\big(1 + |A|\big).
\end{equation}
If $\|A\|_1<1$ then 
\[
 \det(1+A) = \exp\left(\sum_{n=1}^\infty \frac{(-1)^{n+1}}{n}\Tr A^n\right).
\]
Applied to the transfer operators $\TO_{s,\varrho}$ as in Section~\ref{sec:TO} we have, for $\Rea s >\delta$, the identity
\[
 \det\left(1 - \TO_{s,\varrho}\right) = \exp\left( - \sum_{n=1}^\infty \frac{1}{n} \Tr\TO_{s,\varrho}^n\right).
\]

\subsection{Further notation} For $ z\in \mathbb{C} $ we set $ \langle z\rangle:= \sqrt{1+\vert z\vert^{2}}$. For any $z\in\CC$ and $R>0$ we let $B(z;R)$ denote the open Euclidean ball in $\CC$ with center $z$ and radius $R$, and we let $\overline{B}(z;R)$ denote its closure. We use the standard conventions of the $O$-notation and for $\ll$, $\gg$, and $\asymp$. In particular, we write $ f(x)\ll g(x) $ ($ f(x)\gg g(x) $) if there exists a constant $ C>0 $ such that $ \vert f(x)\vert\leq C \vert g(x)\vert $ ($ \vert f(x)\vert\geq C \vert g(x)\vert $) for all considered $x$. Further, $ f(x)\asymp g(x) $ if $ f(x)\ll g(x) $ and $ f(x)\gg g(x) $.

\section{Proof of Proposition~\ref{prop:twist_growth}}\label{sec:proofprop}

In this section we provide a proof of Proposition~\ref{prop:twist_growth}. In case that the representation considered in Proposition~\ref{prop:twist_growth} is the trivial character $\trivrep_\CC$, Proposition~\ref{prop:twist_growth} is identical to \cite[Proposition~3.2]{Guillope_Lin_Zworski}. For the proof of Proposition~\ref{prop:twist_growth} for general finite-dimensional unitary representations $\varrho$ we combine the proof by Guillop\'e--Lin--Zworski of \cite[Proposition~3.2]{Guillope_Lin_Zworski} with the observation that the contributions of the representation $\varrho$ in the transfer operator and the dynamical parts can be separated, which allows us to reduce the study to the transfer operator with the trivial character. We carefully show that all necessary estimates for the proof of Proposition~\ref{prop:twist_growth} are indeed uniform for all considered finite-dimensional unitary representations.

Throughout this section let $\Gamma$ be a Schottky group. We use the notation from Section~\ref{sec:prelims}. In particular, we let $\mc D_1,\ldots, \mc D_{2m}$ denote the open disks in $\CC$ and $\gamma_1,\ldots, \gamma_{2m}$ the generators (already including the inverses) of $\Gamma$ used in the geometric construction of $\Gamma$, and we use the Hilbert Bergman space from Section~\ref{sec:TO}.

\begin{proof}[Proof of Proposition~\ref{prop:twist_growth}]
Let $V$ be a finite-dimensional unitary space, $\varrho\colon\Gamma\to U(V)$ a unitary representation of $\Gamma$, and let $\TO_{s,\varrho}$ denote the transfer operator associated to $(\Gamma,\varrho)$ (see \eqref{TO_def}). We consider $\TO_{s,\varrho}$ as an operator of the Hilbert Bergman space defined in Section~\ref{sec:TO}. Recall from \eqref{FredholIdentity} that 
\[
 L_\Gamma(s,\varrho) = \det\big(1-\TO_{s,\varrho}\big).
\]
For all $s\in\CC$ the Weyl inequality (see \eqref{detsing}) implies that 
\[
|L_\Gamma(s,\varrho)| = |\det\big(1-\TO_{s,\varrho}\big)| \leq \det(1 + |\TO_{s,\varrho}|).
\]
In the following we estimate the right hand side further from above. To that end we consider the operator
\[
 U\sceq \bigoplus_{j=1}^{2m} \varrho(\gamma_j)
\]
which acts on $\mc H$ by 
\[
 Uf= \bigoplus_{j=1}^{2m}\varrho(\gamma_j)f_j
\]
for all $f=\bigoplus_{j=1}^{2m} f_j\in \mc H$. Then
\[
 \TO_{s,\varrho} = \TO_{s,\trivrep_V}\circ U.
\]
Since $U$ is obviously unitary, $U^* = U^{-1}$, and
\[
|\TO_{s,\varrho}| = U^{-1}\circ |\TO_{s,\trivrep_V}|\circ U.
\]
In turn (see \eqref{def_det})
\[
\det(1+|\TO_{s,\varrho}|) = \det(1+|\TO_{s,\trivrep_V}|).
\]
Let $I_V$ denote the identity operator on $V$. From $\TO_{s,\trivrep_V} = \TO_{s,\trivrep_\CC} \otimes I_V$ it follows that
\[
 |\TO_{s,\trivrep_V}| = |\TO_{s,\trivrep_\CC}|\otimes I_V.
\]
Thus, we have
\[
\det\big( 1 + |\TO_{s,\trivrep_V}|\big) = \det(1+|\TO_{s,\trivrep_\CC}|)^{\dim V}.
\]
By \eqref{def_det},
\[
 \det(1+|\TO_{s,\trivrep_\CC}|) = \prod_{k=1}^\infty (1+\mu_k(\TO_{s,\trivrep_\CC})).
\]
By \cite[Proof of Proposition~3.2]{Guillope_Lin_Zworski} there exist constants $c_1,c_2>0$ (only depending on $\Gamma$) such that for all $s\in\CC$ and all $k\in\N$ we have
\[
 \mu_k(\TO_{s,\trivrep_\CC}) \leq c_1e^{c_1|s|-c_2k}.
\]
Thus,
\[
 \prod_{k=1}^\infty (1+\mu_k(\TO_{s,\trivrep_\CC})) \leq \prod_{k=1}^\infty \left( 1 + c_1e^{c_1|s|-c_2k} \right)
\]
Let 
\[
 \ell(s)\sceq \left\lceil\frac1{c_2}\left( \log c_1 + c_1|s| \right)\right\rceil.
\]
Then 
\[
 \prod_{k=\ell(s)+1}^\infty \left( 1 + c_1e^{c_1|s|-c_2k} \right) \leq \prod_{m=1}^\infty \left( 1 + e^{-c_2m}\right),
\]
which is convergent and bounded independently of $s$.

Further (note that $e^{c_1|s|}\geq 1$ for the second inequality)
\begin{align*}
\prod_{k=1}^{\ell(s)} \left( 1 + c_1e^{c_1|s|-c_2k} \right) & \leq \left( 1 + c_1e^{c_1|s|}\right)^{\ell(s)}
\\
& \leq \left( c_3 e^{c_1|s|}\right)^{\ell(s)}
\\
& \leq \exp\left( c_3 + c_4 |s| + c_5 |s|^2\right)
\\
& \leq \exp\left( c_6 + c_7 |s|^2 \right)
\end{align*}
with appropriate constants $c_3,\ldots, c_7>0$ (only depending on $\Gamma$). Thus, there exists $c_8>0$ such that 
\[
 |L_\Gamma(s,\varrho)| \leq \left( c_8 e^{c_6 + c_7|s|^2} \right)^{\dim V}.
\]
It follows that
\[
 \log |L_\Gamma(s,\varrho)| \ll \dim V \cdot \langle s\rangle^2.
\]
This completes the proof of Proposition~\ref{prop:twist_growth}. 
\end{proof}

\section{Proof of Theorem~\ref{thm:counting}}\label{sec:proof1}

In this section we prove Theorem~\ref{thm:counting}. Throughout let $X$ be a Schottky surface and let $\Gamma$ be a Schottky group such that $X=\Gamma\backslash\hh$.

We first consider the case that $X$ is elementary, hence a hyperbolic cyclinder. Then $\Gamma$ 
is generated by a single hyperbolic element, say $ \Gamma = \langle \gamma\rangle$. In this case the resonances of $ X $ can be computed explicitly (see, e.\,g., \cite[Proposition~5.2]{Borthwick_book}). The counting function satisfies the asymptotic formula  
$$ N_{X}(r) \sim \frac{\ell(\gamma)}{2}r^{2}.$$

If $ \widetilde{X} $ is a cover of $ X $ of degree $ k $ then $ \widetilde{X} = \langle \gamma^{k} \rangle\setminus \mathbb{H}^{2} $. Hence
$$ N_{X}(r) \sim \frac{\ell(\gamma^{k})}{2}r^{2} = k\frac{\ell(\gamma)}{2}r^{2}, $$
which establishes an even stronger result than Theorem~\ref{thm:counting}.

We suppose from now on that $X$ is non-elementary. We first prove the upper bound stated in Theorem~\ref{thm:counting}\eqref{counti}. This proof is based on the following two key ingredients: Suppose that $\wt X = \wt\Gamma\backslash\hh$ is a finite cover of $X$, or equivalently, suppose that $\wt\Gamma\subseteq\Gamma$ is a subgroup of finite index. As briefly discussed in Section~\ref{sec:SZF}, the resonance counting function $N_{\wt X}(r)$ can be bounded from above by the number of zeros of the Selberg zeta function $Z_{\wt\Gamma}$ in $\{ |s|\leq r\}$. By the Venkov--Zograf factorization formula (which follows directly from \cite{Venkov_Zograf} in combination with the simplification from \cite[Theorem~7.2]{Venkov_book}, or as a special case of \cite[Theorem~6.1]{FP_szf}), the Selberg zeta function $Z_{\wt\Gamma}$ of $\wt X$ is identical to the $L$-function of $(\Gamma,\lambda)$, where   
\[
 \lambda = \Ind_{\wt\Gamma}^\Gamma \trivrep_{\wt\Gamma}
\]
is the representation of $\Gamma$ obtained from the induction of the one-dimensional trivial representation of $\wt\Gamma$. Thus
\[
 Z_{\wt\Gamma}(s) = L_{\Gamma}(s,\lambda).
\]
Proposition~\ref{prop:twist_growth} allows us to bound $L_{\Gamma}$ (and hence $Z_{\wt\Gamma}$) in terms of 
\[
 \dim\lambda = [\Gamma:\wt\Gamma] = \dcov(\wt X,X)
\]
and additional factors that are independent of $\wt\Gamma$. These estimates result in an upper bound for $N_{\wt X}(r)$.

The lower bound stated in Theorem~\ref{thm:counting}\eqref{countii} is then shown by using the upper bound in combination with the so-called Guillop\'e--Zworski argument \cite{GZ_scattering_asympt, GZ_Wave}. As for Proposition~\ref{prop:twist_growth} we carefully show that all constants are uniform for all finite covers of $X$.

Throughout we assume without loss of generality that the Schottky group $\Gamma$ is chosen such that the disks  \eqref{disks} used in the geometric construction of $\Gamma$ are contained in $\CC$. Moreover, for any finite cover $\wt X$ of $X$ we choose a representative $\wt\Gamma$ of its fundamental group such that $\wt\Gamma$ is a subgroup of $\Gamma$.

As a preparation we recall Titchmarsh's Number of Zeros Theorem \cite{Titchmarsh_book}, which is a consequence of Jensen's formula. 

\textit{Let $z_0\in\CC$, $T>0$ and let $f\colon \overline{B}(z_0;T)\to\CC$ be a function such that $f$ is bounded on $\overline{B}(z_0;T)$ by, say, $M\geq 0$, analytic on $B(z_0;T)$, and $f(z_0)\not=0$. Then, for all $\eta\in (0,1)$, the number of zeros (with multiplicities) of $f$ in $\overline{B}(z_0;\eta T)$ is at most
\[
 \frac{1}{\log\frac1\eta}\left( \log M - \log |f(z_0)|\right).
\]
}

Note that if $f$ extends analytically to a neighborhood of  $\overline{B}(z_0;T)$ then $M$ is attained at the boundary of $\overline{B}(z_0;T)$.

\begin{proof}[Proof of Theorem~\ref{thm:counting}\eqref{counti} (upper bound)]
All constants $ c_{n} $ with $ n\in \{ 1, 2, \dots \} $ that appear during the proof are positive and may depend on $ \Gamma $ (or equivalently, on $ X $). None of these constants depend on any finite cover of $X$.

Let $\wt X=\wt\Gamma\backslash\hh$ be a finite cover of $X$, let $\trivrep_{\wt\Gamma}\colon\wt\Gamma\to \mathbb{S}^1$ denote the trivial character of $\wt\Gamma$, and let
\[
\lambda := \Ind_{\wt\Gamma}^\Gamma \trivrep_{\wt \Gamma}
\]
denote its induction to a representation of $\Gamma$. Let $s\in\CC$. By the Venkov--Zograf factorization formula  \cite{Venkov_Zograf, Venkov_book, FP_szf} we have 
$$
 Z_{\wt\Gamma}(s) = L_{\wt\Gamma}(s,\trivrep_{\wt\Gamma}) = L_\Gamma(s,\lambda).
$$
Recall that
$$ \dim \lambda = [ \Gamma : \wt \Gamma ] = \dcov(\wt X,X). $$ 
From Proposition \ref{prop:twist_growth} it follows that
\begin{equation}
\log \left\vert Z_{\wt \Gamma}(s) \right\vert \leq c_{1} \dcov(\wt X,X) \langle s \rangle^{2}.
\end{equation}
In order to convert the growth estimate for $ Z_{\wt \Gamma} $ into an upper bound for the number of resonances, we note that
\begin{align*}
N_{\wt X}(r) &\leq \#\{ s\in\CC \colon |s|\leq r,\ Z_{\wt \Gamma}(s) = 0\}
\\
& \leq \# \{s\in\CC \colon |s-1|\leq r+1,\ Z_{\wt \Gamma}(s) = 0\}.
\end{align*}
Since $L_\Gamma(\cdot,\lambda) = Z_{\wt\Gamma}$ is analytic on all of $\CC$, and $ Z_{\wt\Gamma}(1) >0 $, Titchmarsh's Number of Zeros Theorem with $z_0=1$, $T=2(r+1)$ and $\eta=1/2$ yields
\begin{align*}
 N_{\wt X}(r) &\leq \frac{1}{\log 2} \left( \log \max \left\{ |Z_{\wt \Gamma}(s)| :  |s-1| =  2(r+1) \right\} - \log Z_{\wt \Gamma}(1)\right)\\
 &\leq  \frac{1}{\log 2} \left( c_{2} \dcov(\wt X,X) \langle 2(r+1) \rangle^2 - \log Z_{\wt \Gamma}(1)\right).
\end{align*}
Since $ r\geq 1 $, we have $ \langle 2(r+1) \rangle^2\ll \langle r \rangle^2 \ll r^2 $, and hence 
\begin{align*}
 N_{\wt X}(r) \leq c_{3} \left( \dcov(\wt X, X)r^{2} - \log Z_{\wt \Gamma}(1) \right).
\end{align*}
We claim that
\[
 Z_{\wt \Gamma}(1) \geq Z_{\Gamma}(1)^{\dim \lambda}.
\]
Indeed, since $ 1>\delta $, the expression of $L$-functions as Euler products applies and yields 
\begin{align*}
Z_{\wt \Gamma}(1) &= L_{\Gamma}(1,\lambda)\\
&= \prod_{\overline g\in\overline\Gamma_p} \prod_{k=0}^\infty \det\left(1-\lambda(g)N_G(g)^{-(1+k)}\right)\\
& \geq \prod_ {\overline g\in\overline\Gamma_p}\prod_{k=0}^\infty \left(1-N_G(g)^{-(1+k)}\right)^{\dim\lambda}\\
& = L_{\Gamma}(1,\trivrep)^{\dim\lambda}\\
&= Z_{\Gamma}(1)^{\dim \lambda}.
\end{align*}
Thus, 
\[
 -\log Z_{\wt \Gamma}(1) \leq  -\dim\lambda \cdot \log Z_{\Gamma}(1) = -\dcov(\wt X, X) \log Z_{\Gamma}(1)  
\]
Clearly, $ -\log Z_{\Gamma}(1) $ is a positive constant only depending on $\Gamma $. We conclude that there exists $ c_{4} > 0 $ such that
\[
 N_{\wt X}(r)  \leq c_{4} \dcov(\wt X,X) r^2.
\]
for all $ r\geq 1 $.
\end{proof}

Taking advantage of the already established upper bound for the resonance counting function, we can now prove the lower bound.

\begin{proof}[Proof of Theorem~\ref{thm:counting}\eqref{countii} (lower bound)]
As in the proof of Theorem \ref{thm:counting}\eqref{counti}, the constants $ c_{n} $ with $ n\in \{ 1, 2, \dots \} $ are all positive and may depend on $ X $, but are independent of any finite cover of $X$.

We take advantage of the following \emph{wave $ 0 $-trace formula} provided by Guillop\'e--Zworski \cite{GZ_Wave}: For any function $ \varphi\in C_{c}^{\infty} \big( (0, \infty)\big) $ let 
\[
\wh\varphi(z) \sceq \int_{-\infty}^\infty e^{-ix z}\varphi(x)\, dx.
\]
be its Fourier transform. Then, for any non-elementary Schottky surface $Y$ and all test functions $\varphi\in C^\infty_c\big( (0,\infty) \big)$ we have
\begin{align}\label{wave_trace}
 \sum_{s\in\mc R(Y)} \wh\varphi\left( i\left(s-\frac12\right) \right)&  = -\frac{\zvol(Y)}{4\pi} \int_{-\infty}^\infty \frac{\cosh\frac{t}2}{\sinh^2 \frac{t}2}\varphi(t)\,dt 
 \\
 & \qquad \qquad
 + \sum_{\ell\in\mc L(Y)}\sum_{k=1}^\infty \frac{\ell}{2\sinh\frac{k\ell}2}\varphi(k\ell), \nonumber
\end{align}
where $ \mc L(Y) $ is the \emph{primitive length spectrum} of $ Y $, that is, the set of lengths of the primitive periodic geodesics on $ Y $ (with multiplicities).

Let $\wt X$ be a finite cover of $X$. Pick $\varphi_1\in C^\infty_c\big( (0,\infty)\big)$ such that $\varphi_1$ is nonnegative and $\supp\varphi_1\subseteq (0,\ell_0(X))$. For $T\in\RR$, $T>0$, we define $\varphi_T\in C_c^\infty\big( (0,\infty) \big)$ by
\[
 \varphi_T(x)\sceq T\varphi_1(Tx).
\]
We want to apply the wave $ 0 $-trace formula to $ \widetilde{X} $ and $ \varphi_{T} $ with $ T\geq 1. $ Note that $\supp\varphi_T\subseteq (0,\ell_0(X)/T)$. 

Since $\ell_0(\wt X) \geq \ell_0(X)$, the sum on the right hand side of \eqref{wave_trace} vanishes for all $ T\geq 1 $:
\[
 \sum_{\ell\in\mc L(\wt X)} \sum_{k=1}^\infty \frac{\ell}{2\sinh\frac{k\ell}{2}}\varphi_T(k\ell) = 0.
\]
Thus
\begin{equation}\label{application_of_trace}
\left| \sum_{s\in\mc R(\wt X)} \wh\varphi_T\left( i\left(s-\frac12\right)\right)\right| = \frac{\zvol(\wt X)}{4\pi} \int_{-\infty}^\infty \frac{\cosh\frac{t}2}{\sinh^2\frac{t}2}\varphi_T(t)\,dt.
\end{equation}
In the following we estimate \eqref{application_of_trace} from above and below. For the lower bound we note that
$$
\int_{0}^{\infty} \frac{\cosh(t/2)}{\sinh(t/2)^{2}}\varphi_{T}(t) dt = \int_{0}^{\ell_{0}(X)} \frac{\cosh(t/2T)}{\sinh(t/2T)^{2}}\varphi_{1}(t) dt. 
$$
From $ \cosh(t/2T) \geq 1 $ and 
$$
\sinh(t/2T) = \sum_{k=1}^{\infty} \frac{(t/2T)^{2k+1}}{(2k+1)!} \leq \frac{1}{T} \sum_{k=1}^{\infty} \frac{(t/2)^{2k+1}}{(2k+1)!} = \frac{1}{T} \sinh(t/2)
$$
for all $ t > 0 $ (recall that $ T\geq 1 $) it follows that
$$
\int_{0}^{\infty} \frac{\cosh(t/2)}{\sinh(t/2)^{2}}\varphi_{T}(t) dt \geq T^{2} \int_{0}^{\ell_{0}(X)} \frac{1}{\sinh(t/2)^{2}}\varphi_{1}(t) dt. 
$$
Thus, \eqref{application_of_trace} can be bounded from below by 
\begin{equation}\label{lower_bound_trace}
\left\vert \sum_{\zeta \in \mathcal{R}(\widetilde{X})} \widehat{\varphi_{T}}\left( i(\zeta-\frac{1}{2})\right)\right\vert \geq c_{1} \text{0-vol}(\widetilde{X}) T^{2},
\end{equation}
with
$$
c_{1} := \frac{1}{4\pi} \int_{0}^{\ell_{0}(X)} \frac{1}{\sinh(t/2)^{2}}\varphi_{1}(t) dt \in (0, \infty).
$$

For an upper bound of \eqref{application_of_trace} we let $ r \geq 1 $, split the sum in the left hand side of \eqref{application_of_trace} at $r$, and estimate 
$$
 \left| \sum_{s\in\mc R(\wt X)} \wh\varphi_T\left( i\left(s-\frac12\right)\right)\right| \leq  \sum_{\substack{s\in\mc R(\wt X) \\ |s|\leq r}} \left| \wh\varphi_T\left( i\left(s-\frac12\right)\right)\right| + \sum_{\substack{s\in\mc R(\wt X)\\ |s|>r}}\left| \wh\varphi_T\left( i\left(s-\frac12\right)\right)\right|.
$$
We estimate both sums on the right hand side separately.

Since $ \varphi_{1}\in C_{c}^{\infty}\big( (0,\ell_{0}(X))\big) $, iterated integration by parts yields
\begin{equation}\label{PW_estimate}
 \left|\wh\varphi_T(z)\right| = \left| \wh\varphi_1\left(\frac{z}{T}\right)\right| \leq c \left( 1 + \left|\frac{z}{T}\right|\right)^{-3} \times
 \begin{cases}
  \exp\left(\frac{\ell_0(X)}{T}\Ima z\right) & \text{if $\Ima z\geq 0$}
  \\
  1 & \text{if $\Ima z\leq 0$,}
 \end{cases}
\end{equation}
for all $ z\in \mathbb{C} $ and $ T > 0 $, where $c>0$ is a constant depending only on $\ell_0(X)$ and the choice of $\varphi_1$.

Recall that for each resonance $s\in \mc R(\wt X)$ we have
\[
 \Im\left( i\left(s-\frac12\right)\right) = \Rea s - \frac12 \leq \delta -\frac12.
\]
From \eqref{PW_estimate} it follows that
\[
 \left| \wh\varphi_T\left( i\left(s-\frac12\right)\right)\right| \leq  c\left( 1 + \left|\frac{s-\frac12}{T}\right|\right)^{-3} \leq c_{2}.
\]
Thus,
\begin{equation}\label{two}
 \sum_{\substack{s\in \mc R(\wt X) \\ |s|\leq r}} \left| \wh\varphi_T\left( i\left(s-\frac12\right)\right)\right| \leq  c_{2} N_{\wt X}(r).
\end{equation}
Using \eqref{PW_estimate} again, we find
\begin{equation}\label{second_partial_sum}
 \sum_{\substack{s\in\mc R(\wt X) \\ |s|>r}}\left| \wh\varphi_T\left( i\left(s-\frac12\right)\right)\right| 
 \leq c \sum_{\substack{s\in\mc R(\wt X) \\ |s|>r}} \left( 1 + \left| \frac{s-1/2}{T} \right|\right)^{-3}.
\end{equation}
The sum on right hand side of \eqref{second_partial_sum} can be bounded by a Stieltjes integral as follows:
\begin{align*}
 \sum_{\substack{s\in\mc R(\wt X) \\ |s|>r}} \left( 1 + \left| \frac{s-1/2}{T} \right|\right)^{-3} & \leq \sum_{\substack{s\in\mc R(\wt X) \\ |s|>r}} \left( 1 - \frac{1}{2T} + \left| \frac{s}{T} \right|\right)^{-3}
 \\
 & \leq T^3\sum_{\substack{ s\in\mc R(\wt X) \\ |s|>r}} |s|^{-3}
 \\
 & \leq T^3\int_r^\infty \frac{1}{t^3} \, dN_{\wt X}(t).
\end{align*}
Note that the integral converges, since $ N_{\widetilde{X}}(t) = O(t^{2}) $ as $ t\to \infty $.

By Theorem~\ref{thm:counting}\eqref{counti} (which is already proven above) there exists $ C > 0 $ (independent of $ \wt X $) such that $ N_{\wt X}(r) \leq C\zvol(\wt X) r^2 $ for all $ r\geq 1 $. (Here we use the relation $ \zvol(\wt X) = \dcov(\wt X, X) \zvol(X) $.) It follows that
\begin{align*}
\int_r^\infty \frac{1}{t^3} \, dN_{\wt X}(t) & = \lim_{R\to\infty} R^{-3} N_{\wt X}(R) - r^{-3} N_{\wt X}(r) + 3\int_r^\infty \frac{N_{\wt X}(t)}{t^4}\, dt
\\
& \leq r^{-3} N_{\wt X}(r) + 3C\zvol(\wt X) \int_r^\infty \frac{dt}{t^2}
\\
& \leq 4C\zvol(\wt X) r^{-1}.
\end{align*}
Thus, we have established  
\begin{equation}\label{three}
\sum_{\substack{s\in\mc R(\wt X) \\ |s|>r}}\left| \wh\varphi_T\left( i\left(s-\frac12\right)\right)\right| 
 \leq c_{3} \zvol(\wt X) T^{3} r^{-1}
\end{equation}
for all $ r\geq 1 $ and $ T \geq 1 $, where $ c_{3} := 4C \cdot c. $

Gathering \eqref{application_of_trace}, \eqref{two} and \eqref{three} leads to the inequality
\begin{equation}\label{low_total}
 c_1 \zvol(\wt X) T^2 \leq c_2 N_{\wt X}(r) + c_3 \zvol(\wt X) T^3 r^{-1},
\end{equation}
which is valid for all $ r \geq 1 $ and $ T \geq 1 $.

Finally set $ a := (2c_{3})^{-1}c_{1} > 0 $ and $ r_{0} := \max \{1, a^{-1}\} $, and notice that these constants only depend on $ X $. For all $ r \geq r_{0} $ we apply \eqref{low_total} with $ T := a r \geq 1 $ to obtain
$$
N_{\wt X}(r) \geq c_{4} \zvol(\wt X) r^{2},
$$
where
$$
c_{4} := \frac{c_{1}a^{2}-c_{3}a^{3}}{c_{2}} = \frac{c_{1}^{3}}{8 c_{2}c_{3}^{2}}.
$$
This completes the proof of Theorem \ref{thm:counting}\eqref{countii}. 
\end{proof}

\section{Proof of Theorem~\ref{thm:box}}\label{sec:proof2}

In this section we provide a proof of Theorem~\ref{thm:box}. This proof follows a route similar to the one taken by Jakobson and Naud for the proof of \cite[Theorem~1.3]{JN}. The main novelties here are an intensive exploitation of twisted transfer operators and a rather detailed study of the fine structure of powers of the transfer operators.

Throughout this section let 
\begin{equation}\label{space}
X=\Gamma\backslash\hh
\end{equation}
be a fixed Schottky surface, and $\delta=\delta(X) = \dim\Lambda(\Gamma)$ the Hausdorff dimension of the limit set of $\Gamma$.

For any finite cover $\wt X = \wt\Gamma\backslash\hh$ we can estimate the number $M_{\wt X}(\sigma, T)$ of resonances of $\wt X$ in the box
\[
 R(\sigma,T)\sceq [\sigma,\delta] + i[T-1,T+1]
\]
by counting the number of zeros of the Selberg zeta function $Z_{\wt\Gamma}$ in $R$, and we can use the identities
\begin{equation}\label{shifting}
 Z_{\wt\Gamma}(s) = L_\Gamma(s,\lambda) = \det\big(1 - \TO_{s,\lambda}\big),
\end{equation}
where $\lambda=\Ind_{\wt\Gamma}^\Gamma \trivrep_{\wt\Gamma}$ is the induction of the trivial character of $\wt\Gamma$ to $\Gamma$, and $\TO_{s,\lambda}$ is the transfer operator associated to $\Gamma$ twisted with $\lambda$. These identities allow us to transfer the counting problem to the transfer operator and to get bounding constants uniform in $\wt X$. 

However, to get better bounds, instead of using $\TO_{s,\lambda}$ we will use a suitable power of this transfer operator. We note that if $1$ is an eigenvalue of $\TO_{s,\lambda}$ of algebraic multiplicity $m$ then, for any $N\in\N$, the value $1$ is an eigenvalue of $\TO_{s,\lambda}^N$ of algebraic multiplicity at least $m$. Thus, in any subset $M$ of $\CC$, the number of zeros $s$ in $M$ of $\det(1-\TO_{s,\lambda})$ is bounded above by the number of zeros in $M$ of $\det(1-\TO_{s,\lambda}^N)$. In particular, for any $N\in\N$, 
\[
 M_{\wt X}(\sigma, T) \leq \#\{ s\in R(\sigma, T) \colon \det(1-\TO_{s,\lambda}^N)=0\}.
\]
As domain of definition for the powers of $\TO_{s,\lambda}$ we use the Hilbert space defined in Section~\ref{sec:refine} below, which can be seen as a refinement of the Hilbert spaces from Section~\ref{sec:TO}.

Throughout let $\Gamma$ be chosen such that the disks $\mc D_1,\ldots, \mc D_{2m}$ (see \eqref{disks}) used in the geometric construction of $\Gamma$ are all contained in $\CC$, let $\gamma_1,\ldots, \gamma_m$ be the associated generators of $\Gamma$ (see Section~\ref{sec:Schottky}), and set 
\[
 \mc D \sceq \bigcup_{j=1}^{2m} \mc D_j.
\]

\subsection{Refined Hilbert spaces and iterates of transfer operators}\label{sec:refine}

We recall from \cite{Guillope_Lin_Zworski} the definition of a family of Hilbert spaces, depending on a parameter $h>0$, which we use as domain of definition for appropriate powers of the transfer operators. 

Throughout let $\Lambda\sceq \Lambda(\Gamma)$ denote the limit set of $\Gamma$. For $h>0$ we let
\[
 \Lambda(h) \sceq (-h,h) + \Lambda.
\]
By \cite{Guillope_Lin_Zworski} we find $h_0>0$ such that for all $h\in (0,h_0)$, the set $\Lambda(h)$ is bounded, has finitely many connected components, say $N(h)$ many, its connected components
\[
 I_p(h), \quad p=1,\ldots, N(h),
\]
are intervals of lengths at most $Ch$ for some $C>0$ independent of $h$, each connected component is contained in some connected component of $\mc D$, and 
\[
 N(h) = O(h^{-\delta}) \quad\text{as $h\searrow 0$,}
\]
where $\delta=\delta(X) = \dim\Lambda$ is the Hausdorff dimension of $\Lambda$.

For each $h\in (0,h_0)$ and $p\in\{1,\ldots, N(h)\}$ let $\mc E_p(h)$ be the open Euclidean disk in $\CC$ with center in $\RR$ such that 
\[
 \mc E_p(h)\cap\RR = I_p(h),
\]
and let 
\[
 \mc E(h) \sceq \bigcup \mc E_p(h).
\]
For each finite-dimensional unitary space $V$ let $H^2(\mc E_p(h);V)$ denote the Hilbert Bergman space of $V$-valued functions on $\mc E_p(h)$, and let 
\[
 H^2(\mc E(h);V) \sceq \bigoplus_{p=1}^{N(h)} H^2(\mc E_p(h);V).
\]
A slight adaptation of \cite{Guillope_Lin_Zworski} shows that there exists $N_1\in\N$ (independent of $h\in (0,h_0)$) such that for all finite-dimensional unitary spaces $V$, all unitary representations $\varrho\colon\Gamma\to U(V)$ and all $N\geq N_1$, the $N$-th power of $\TO_{s,\varrho}$ defines on operator on $H^2(\mc E(h);V)$:
\[
 \TO_{s,\varrho}^N\colon H^2(\mc E(h);V) \to H^2(\mc E(h);V).
\]
Considered as an operator on $H^2(\mc E(h);V)$, the transfer operator $\TO_{s,\varrho}^N$ remains to be of trace class, and its Fredholm determinant is identical to the one of $\TO_{s,\varrho}^N$ as an operator on the Hilbert space from Section~\ref{sec:TO}.

A rather precise formula for the iterates of $\TO_{s,\varrho}$ can be given. To that end, for any $N\in\N$ and any $\alpha = (\alpha_1,\ldots,\alpha_N)\in \{1,\ldots, 2m\}^N$ we set
\[
 \gamma_\alpha \sceq \gamma_{\alpha_1}\cdots \gamma_{\alpha_N}.
\]
Further, we let
\[
 \mc W_N\sceq \big\{ (\alpha_1,\ldots,\alpha_N) \in \{1,\ldots, 2m\}^N : \forall\, j\in\{1,\ldots, N-1\}\colon \alpha_{j+1}\not=\alpha_j +m \mod 2m\big\}
\]
denote the set of elements in $\{1,\ldots, 2m\}^N$ that correspond to the elements in $\Gamma$ of minimal word length $N$ over the alphabet $\{\gamma_1,\ldots,\gamma_{2m}\}$. For $j\in \{1,\ldots, 2m\}$ we set
\[
 \mc W_N^j \sceq \{ \alpha\in\mc W_N : \alpha_1\not= j+m \mod 2m\}.
\]
A straightforward induction shows that 
\[
 \TO_{s,\varrho}^N = \sum_{j=1}^{2m} 1_{\mc D_j} \sum_{\alpha\in\mc W_N^j} \nu_s(\gamma_\alpha).
\]

\subsection{Separation Lemmas}\label{sec:separation}

Given $z\in \mc D_j$ and $\alpha,\beta\in\mc W_j^N$ for some $j\in\{1,\ldots, 2m\}$ and $N\in\N$, the images of $z$ under $\gamma_\alpha^{-1}$ and $\gamma_\beta^{-1}$ can be rather close to each other. In this section we discuss under which conditions we know that then already $\gamma_\alpha=\gamma_\beta$. These results are crucial for growth bounds of the Fredholm determinant of $\TO_{s,\varrho}^N$, see Proposition~\ref{prop:NTX} below.

Throughout, $X$ refers to the fixed Schottky surface \eqref{space}. All (implied) constants may depend on $X$.

\begin{lem}\label{lem:near}
Let $C>0$. Then there exists $h_1\in (0,1)$ (depending on $X$ and $C$) and $C_1>0$ (depending on $X,C,h_1$) such that for all $j\in\{1,\ldots, 2m\}$, for all $z\in\mc D_j$, for all $h\in (0,h_1)$, for all $N\in\N$ with $N<C_1 \log h^{-1}$ and for all $\alpha,\beta\in\mc W_N^j$ the bound
\[
 \left| \gamma_\alpha^{-1}.z-\gamma_\beta^{-1}.z\right| < Ch
\]
implies $\alpha=\beta$.
\end{lem}

\begin{proof}
By \cite[Lemma~4.4]{JN} we find $c>0$ and $\varrho\in (0,1)$ such that for all $j\in\{1,\ldots, 2m\}$, for all $z\in\mc D_j$, for all $N\in\N$, for all $\alpha,\beta\in\mc W_N^j$ with $\alpha\not=\beta$ we have
\begin{equation}\label{JN_cr}
 \left|\gamma_\alpha^{-1}.z - \gamma_\beta^{-1}.z \right| \geq c\varrho^N.
\end{equation}
Let $C>0$. Suppose that we have $h\in (0,1)$, $j\in\{1,\ldots, 2m\}$, $z\in\mc D_j$, $N\in\N$, $\alpha,\beta\in\mc W_N^j$ such that $\alpha\not=\beta$ and 
\begin{equation}\label{near}
 \left| \gamma_\alpha^{-1}.z - \gamma_\beta^{-1}.z\right| < Ch.
\end{equation}
Combining \eqref{JN_cr} and \eqref{near} yields $c\varrho^N < Ch$, thus (note that $\log\varrho<0$)
\[
 N > c_1 + c_2  \log h^{-1}.
\]
where
\[
 c_1 \sceq \frac{\log C - \log c}{\log\varrho},\qquad c_2 \sceq \frac{1}{\log \varrho^{-1}}.
\]
Pick $h_0\in (0,1)$ such that 
\[
  c_2 > \frac{c_1}{\log h_0^{-1}}
\]
and pick 
\[
 C_1 \in \left(0 , c_2 - \frac{c_1}{\log h_0^{-1}}\right).
\]
Thus,
\[
 N > C_1\log h^{-1}.
\]
This completes the proof of the lemma.
\end{proof}

\begin{lem}\label{lem:smallconj}
There exists $C_2>0$ (depending on $X$) such that for all finite covers $\wt X=\wt\Gamma\backslash\hh$ of $X$, for all $N\in\N$ with $N<C_2\ell_0(\wt X)$ and for all $\alpha,\beta\in \mc W_N$ with $\Tr\Ind_{\wt\Gamma}^\Gamma \trivrep_{\wt \Gamma}(\gamma_\alpha\gamma_\beta^{-1}) \not=0$ we have $\alpha=\beta$. 
\end{lem}

\begin{proof}
Let $\|\cdot\|_F$ denote the Frobenius norm on $\SL_2(\RR)$. Let $g=\begin{pmatrix} a & b \\ c & d \end{pmatrix} \in\SL_2(\RR)$ be hyperbolic. Then we have
\begin{align*}
  \left\| g \right\|_F^2 & = \Tr(g^\top g) = a^2 + b^2 + c^2 + d^2 
   = (a+d)^2 + (b-c)^2 - 2 
   \\
   & \geq (\Tr g)^2 - 2 = \left( e^{\ell(g)/2} + e^{-\ell(g)/2}\right)^2 - 2 
   \\
   & \geq e^{\ell(g)}.
\end{align*}
Set 
\[
 K \sceq \max\big\{\|\gamma_j\|_F : j\in\{1,\ldots, 2m\}\big\}
\]
and
$$
C_{2} \sceq \frac{1}{4 \log K}.
$$
Let $\wt\Gamma$ be any subgroup of $\Gamma$ of finite index. We argue by contradiction. Let $N\in\N$ with $ N < C_{2} \ell_0(\wt X) $ and suppose that there exist $\alpha,\beta\in\mc W_N$ such that $\alpha\not=\beta$ and $\Tr\Ind_{\wt\Gamma}^\Gamma \trivrep_{\wt \Gamma}(\gamma_\alpha\gamma_\beta^{-1}) \not=0$. Let 
\[
 h\sceq \gamma_\alpha\gamma_\beta^{-1}.
\]
Since $\alpha,\beta\in\mc W_N$, the element $ h $ is the product of at most $ 2N $ matrices in the set $ \{\gamma_{1},\ldots, \gamma_{2m}\} $, which shows that 
\begin{equation}\label{hupper}
 \| h\|_F \leq K^{2N}.
\end{equation}
Since $\Tr\Ind_{\wt\Gamma}^\Gamma \trivrep_{\wt \Gamma}(h) \not=0$,  there exists $p\in\Gamma$ is such that
\[
 php^{-1} \in \wt\Gamma.
\]
Since $\alpha\not=\beta$, the element $h$ is hyperbolic. Further,  
\[
 \ell(h) = \ell\big(php^{-1}\big) \geq \ell_0\big(\wt X\big).
\]
Thus, 
\begin{equation}\label{hlower}
 \| h\|_F \geq e^{\ell_0(\wt X)/2}.
\end{equation}
Combining \eqref{hupper} and \eqref{hlower} yields
\[
 N\geq \frac{1}{4\log K} \ell_0(\wt X) = C_{2}\ell_0(\wt X),
\]
a contradiction. This completes the proof (note that $K>1$). 
\end{proof}

The combination of Lemmas~\ref{lem:near} and \ref{lem:smallconj} yields the following result.

\begin{lem}\label{lem:vanishing}
Let $C>0$. Then there exists $h_1\in (0,1)$ and $\eps_0>0$ such that for all $h\in (0,h_1)$, for all finite covers $\wt X = \wt\Gamma\backslash\hh$ of $X$, for all $N\in\N$ with 
$N\leq \eps_0\big( \ell_0(\wt X) + \log h^{-1}\big)$, for all $j\in\{1,\ldots, 2m\}$, for all $z\in \mc D_j$, for all $\alpha,\beta\in\mc W_N^j$ the following is satisfied: if 
\[
  \Tr\Ind_{\wt\Gamma}^\Gamma  \trivrep_{\wt \Gamma}\big(\gamma_\alpha\gamma_\beta^{-1}\big) \not=0\quad\text{and}\quad \left|\gamma_\alpha^{-1}.z - \gamma_\beta^{-1}.z \right| < Ch
\]
then $\alpha=\beta$.
\end{lem}

\subsection{Bounds on Fredholm determinants}\label{sec:fredholmbounds}

In this section, we provide growth bounds on the Fredholm determinants of iterates of the transfer operator $\TO_{s,\lambda}$, where $\lambda=\Ind_{\wt\Gamma}^\Gamma\trivrep_{\wt\Gamma}$ for finite covers $\wt X =\wt\Gamma\backslash\hh$ of $X=\Gamma\backslash\hh$. These bounds together with an application of Titchmarsh's Number of Zeros Theorem allow us to prove Theorem~\ref{thm:box}, see Section~\ref{sec:finalproof} below.

Throughout, $X=\Gamma\backslash\hh$ is the fixed Schottky surface \eqref{space}, $\delta=\delta(X)$ denotes the Hausdorff dimension of the limit set of $\Gamma$, and all powers of transfer operators are defined on the Hilbert spaces from Section~\ref{sec:refine} for some $h\in (0,h_0)$. As in Section~\ref{sec:separation}, all constants may depend on $X$.

\begin{prop}\label{last_bound}
There exists a constant $C>0$ such that for all finite covers $\wt X$ of $X$, for all $N\in\N$, for all $s\in\CC$ with $\Rea s > \delta$ we have
\[
 -\log\left| \det\left( 1 - \TO_{s,\lambda}^N \right) \right| \leq C N \dcov(\wt X, X) \frac{\Rea s}{\Rea s -\delta} e^{-(\Rea s-\delta)\ell_0(\wt X)},
\]
where $\lambda\sceq\Ind_{\wt\Gamma}^\Gamma\trivrep_{\wt\Gamma}$ denotes the representation of $\Gamma$ that is induced by the trivial character $\trivrep_{\wt\Gamma}$ of $\wt\Gamma$.
\end{prop}

\begin{proof}
Since $ \mathrm{Re}(s) > \delta$, we have
\begin{align*}
\det\left(  1-\mathcal{L}_{
s,\lambda}^{N} \right) &= \exp\left( - \sum_{n=1}^{\infty} \frac{1}{n}\mathrm{Tr}\, \left(\mathcal{L}_{s,\lambda}^{nN}\right)\right).
\end{align*}
This leads to 
$$
\left\vert \det\left(  1-\mathcal{L}_{s,\lambda}^{N} \right) \right\vert = \exp\left( -\mathrm{Re} \sum_{n=1}^{\infty} \frac{1}{n}\mathrm{Tr}\, \left(\mathcal{L}_{s,\lambda}^{nN}\right) \right),
$$
and therefore 
\begin{equation}\label{eest1}
- \log \left\vert \det\left(  1-\mathcal{L}_{s,\lambda}^{N} \right) \right\vert = \mathrm{Re} \sum_{n=1}^{\infty} \frac{1}{n}\mathrm{Tr}\, \left(\mathcal{L}_{s,\lambda}^{nN}\right) \leq \sum_{n=1}^{\infty} \frac{1}{n}\left\vert \mathrm{Tr}\, \left(\mathcal{L}_{s,\lambda}^{nN}\right) \right\vert.
\end{equation}
By adding extra non-negative terms to the infinite sum, we obtain 
\begin{equation}\label{eest2}
\sum_{n=1}^{\infty} \frac{1}{n}\left\vert \mathrm{Tr}\, \left(\mathcal{L}_{s,\lambda}^{nN}\right) \right\vert = N \sum_{n=1}^{\infty} \frac{1}{Nn}\left\vert \mathrm{Tr}\, \left(\mathcal{L}_{s,\lambda}^{nN}\right) \right\vert \leq N \sum_{m=1}^{\infty} \frac{1}{m}\left\vert \mathrm{Tr}\, \left(\mathcal{L}_{s,\lambda}^{m}\right) \right\vert. 
\end{equation}
Let $ L_{S}(\gamma) $ denote the word length of $ \gamma $ with respect to the generating set $ S=\{\gamma_{1}, \dots, \gamma_{2m}\} $ of $ \Gamma  $.
We denote by $ \mathrm{WL}(\gamma) = \min\{ L_{S}(g) : g\in \overline{\gamma}\} $ the minimal word length of any element in the conjugacy class of $ \gamma $. Analogously to \cite[Proof of Proposition~2.2]{JNS} we find 
\begin{equation}\label{Final_trace}
\mathrm{Tr}( \lt_{s,\lambda}^m) = \sum_{d \mid m} \sum_{\overline{\gamma}\in \overline{\Gamma}_{p} \atop \mathrm{WL}(\gamma) = d} d \chi_{\lambda}(\gamma^{m/d}) \frac{e^{-s \ell(\gamma)\frac{m}{d}}}{1-e^{-\ell(\gamma)\frac{m}{d}}},
\end{equation}
where
\[
\chi_{\lambda}(\gamma) \sceq \Tr\Ind_{\wt\Gamma}^\Gamma  \trivrep_{\wt \Gamma}(\gamma).
\]
Combining \eqref{eest1}--\eqref{Final_trace} yields
$$
- \log \left\vert \det\left(  1-\mathcal{L}_{s,\lambda}^{N} \right) \right\vert \leq N \sum_{m=1}^{\infty} \frac{1}{m} \sum_{d \mid m} \sum_{\overline{\gamma}\in \overline{\Gamma}_{p} \atop \mathrm{WL}(\gamma) = d} d \chi_{\lambda}(\gamma^{m/d}) \frac{e^{- \mathrm{Re}(s) \ell(\gamma)\frac{m}{d}}}{1-e^{-\ell(\gamma)\frac{m}{d}}}.
$$
Introducing the new variable $ k = m/d $ and rearranging the above sum accordingly leads to
\begin{align*}
- \log \left\vert \det\left(  1-\mathcal{L}_{s,\lambda}^{N} \right) \right\vert &\leq N \sum_{k=1}^{\infty} \sum_{\overline{\gamma}\in \overline{\Gamma}_{p}} \frac{1}{k} \chi_{\lambda}(\gamma^{k}) \frac{e^{- \mathrm{Re}(s) \ell(\gamma)k }}{1-e^{-\ell(\gamma)k}}\\ &\leq N \sum_{k=1}^{\infty} \sum_{\overline{\gamma}\in \overline{\Gamma}_{p}} \chi_{\lambda}(\gamma^{k}) \frac{e^{- \mathrm{Re}(s) \ell(\gamma)k }}{1-e^{-\ell(\gamma)k}},
\end{align*}
where we have dropped the $ 1/k $-terms in the last estimate. Further, 
\begin{align*}
-\log \left\vert \det\left(  1-\mathcal{L}_{s,\lambda}^{N} \right) \right\vert & \leq N  \sum_{\overline{\gamma}\in \overline{\Gamma}} \chi_{\lambda}(\gamma) \frac{e^{- \mathrm{Re}(s) \ell(\gamma)}}{1-e^{-\ell(\gamma)}}\\
&\leq  \frac{N}{1-e^{-\ell_{0}(X)}} \sum_{\overline{\gamma}\in \overline{\Gamma}} \chi_{\lambda}(\gamma) e^{- \mathrm{Re}(s) \ell(\gamma)}.
\end{align*}
Recall that $ \chi_{\lambda}(\gamma) > 0 $ implies that $\gamma$ in conjugate to an element in $\wt\Gamma$, and hence $ \ell(\gamma)\geq \ell_{0}\big( \wt X \big)$. Moreover, $\chi_\lambda\leq [\Gamma:\wt\Gamma]$. Thus, 
\begin{equation}\label{consequence_of_Frobenius}
-\log \left\vert \det\left(  1-\mathcal{L}_{s,\lambda}^{N} \right) \right\vert \leq \frac{N [\Gamma : \widetilde{\Gamma}] }{1-e^{-\ell_{0}(X)}} \sum_{\substack{ \overline{\gamma}\in \overline{\Gamma} \\ \ell(\gamma) \geq \ell_{0}(\widetilde{X})}}  e^{- \ell(\gamma)\mathrm{Re}(s)}.
\end{equation}
By the prime geodesic theorem \cite[Corollary~11.2]{Lalley} we find a constant $ C >0 $ such that 
\begin{equation}\label{prime_orbit_estimate}
\Pi_{\Gamma}(t) := \# \{ \overline{\gamma}\in \overline{\Gamma} : \ell(\gamma) \leq t \} \leq Ce^{\delta t}.
\end{equation} 
Interpreting the right hand side of \eqref{consequence_of_Frobenius} as a Stieltjes integral and using (\ref{prime_orbit_estimate}) we obtain 
\begin{align*}
\sum_{\substack{ \overline{\gamma}\in \overline{\Gamma} \\ \ell(\gamma) \geq \ell_{0}(\widetilde{X})}}  e^{- \ell(\gamma)\mathrm{Re}(s)} = \mathrm{Re}(s)\int_{\ell_{0}(\widetilde{X})}^{\infty} e^{-\mathrm{Re}(s)x}  \Pi_{\Gamma}(x)dx \leq C \frac{\mathrm{Re}(s) e^{-(\mathrm{Re}(s)-\delta)\ell_{0}(\widetilde{X})}}{\mathrm{Re}(s)-\delta}.
\end{align*}
This completes the proof of Proposition \ref{last_bound}.
\end{proof}

\begin{prop}\label{prop:NTX}
There exists $\eps_0>0$, $N_0\in\N$, and a map $\eta\colon\RR\to\RR$ that is strictly concave, strictly increasing and has a unique zero at $\delta/2$ such that for each pair $\sigma_1>\sigma_0\geq 0$ and each $T_0\in\RR$ there exists a constant $c>0$ (depending continuously on $T_0$) such that for all $T\in\RR$ and all $s\in (\sigma_0,\sigma_1)+i(T-T_0,T+T_0)$ and each finite cover $\wt X$ of $X$ we have
\[
\log\left| \det\left( 1 - \TO_{s,\lambda}^{2N(T,\wt X)} \right)\right| \leq c \dcov(\wt X,X) e^{-\eta(\sigma_0) \ell_0(\wt X)} \langle T\rangle^{\delta-\eta(\sigma_0)}  
\]
with
\[
 N(T,\wt X) = \left\lfloor \eps_0\left( \ell_0(\wt X) + \log\langle T\rangle \right) + N_0 +1 \right\rfloor,
\]
where $\lambda\sceq\Ind_{\wt\Gamma}^\Gamma\trivrep_{\wt\Gamma}$ denotes the representation of $\Gamma$ that is induced by the trivial character $\trivrep_{\wt\Gamma}$ of $\wt\Gamma$.
\end{prop}

\begin{proof}
Throughout let 
\[
 d\sceq \dim\lambda = \dcov(\wt X,X) = [\Gamma : \wt\Gamma].
\]
and let $V$ be the $d$-dimensional unitary vector space on which $\lambda$ represents $\Gamma$.  

In this proof we consider the iterates of the transfer operator $\TO_{s,\lambda}$ as an operator on the Hilbert space $H^2(\mc E(h);V)$ for a specific $h$, approximately of size $\langle T\rangle^{-1}$. To be more precise, we fix some parameters:
\begin{itemize}
\item  Recall the parameter $h_0>0$ from Section~\ref{sec:refine}. We fix $C>0$ such that for all $h\in (0,h_0)$ and all $p\in\{1,\ldots, N(h)\}$ we have $\diam \mc E_p(h)<Ch$.
\item Depending on the choice of $C$, we fix $h_1=h_1(C) \in (0,h_0)$ and $\eps_0 = \eps_0(C)>0$ such that the conclusions of Lemma~\ref{lem:vanishing} are valid. 
\item For all $h\in (0,h_1)$ and $j\in\{1,\ldots, 2m\}$ let 
\[
 \mc P_j(h)\sceq \{ p\in\{1,\ldots, N(h)\} \colon \mc E_p(h)\subseteq \mc D_j\}.
\]
By \cite[Lemma~3.2]{Naud_inventiones} we find $N_0\in\N$ such that for all $N>N_0$, for all $h\in (0,h_1)$, all $j\in\{1,\ldots,2m\}$, all $\alpha\in\mc W_N^j$, all $p\in\mc P_j(h)$ there exists a unique $q\in\{1,\ldots, N(h)\}$ such that 
\begin{equation}\label{contained}
\gamma_\alpha^{-1}\big(\mc E_p(h)\big) \subseteq \mc E_q(h)\quad\text{and}\quad d\big(\gamma_\alpha^{-1}\big(\mc E_p(h)\big), \partial\mc E_q(h)\big)\geq \frac{h}{2}. 
\end{equation}
Recall the number $N_1\in\N$ from Section~\ref{sec:refine}. We fix $N_0$ such that 
\[
 e^{-(N_0+1)/\eps_0} < h_1
\]
and $N_0>N_1$.
\item Let $T\in\RR$. We set $h\sceq e^{-(N_0+1)/\eps_0}\langle T\rangle^{-1}$.
\item We set $N \sceq N_{(\wt X, T)}\sceq \left\lfloor \eps_0\left( \ell_0(\wt X) + \log h^{-1}\right) \right\rfloor$. Note that
\[
 N \geq \eps_0\log h^{-1} -1 \geq N_0.
\]
\end{itemize}

By \cite[Lemma~3.3]{Simon} and the relation between the trace norm and the Hilbert-Schmidt norm (denoted by $\|\cdot\|_\HS$), for all $s\in\CC$ we find
\begin{equation}\label{HSest}
 \log \left| \det\left( 1 - \TO_{s,\lambda}^{2N}\right)\right| \leq \left\| \TO_{s,\lambda}^{2N}\right\|_1  \leq \left\| \TO_{s,\lambda}^N\right\|_\HS^2.
\end{equation}
In order to find an upper estimate for the Hilbert-Schmidt norm of $\TO_{s,\lambda}^N$ we take advantage of an explicit well-balanced Hilbert basis for the Hilbert Bergman space $H^2(\mc E(h);\CC)$. In order to state this Hilbert basis we fix an orthonormal basis $e_1,\ldots, e_d$ of $V$. 

For $p\in\{1,\ldots, N(h)\}$ let $r_p\sceq r_p(h)$ and $c_p\sceq c_p(h)$ denote the radius and center of $\mc E_p=\mc E_p(h)$, respectively. For $q\in\N_0$ set
\[
 \kappa_{p,q}\sceq \kappa_{p,q}^{(h)}\colon \mc E_p \to \CC, \quad \kappa_{p,q}(z)\sceq \sqrt{\frac{q+1}{\pi r_p^2}} \left(\frac{z-c_p}{r_p}\right)^q.
\]
Then $\{ \kappa_{p,q}\}_{q\in\N_0}$ is a Hilbert basis for $H^2(\mc E_p;\CC)$. We extend each function $\kappa_{p,q}$ to a function $\varphi_{p,q}\colon \mc E\to \CC$ by
\[
\varphi_{p,q}(z) \sceq 
\begin{cases}
\kappa_{p,q}(z) & \text{for $z\in \mc E_p$}
\\
0 & \text{otherwise.}
\end{cases}
\]
For $k\in\{1,\ldots, d\}$, $p\in\{1,\ldots,N(h)\}$, $q\in\N_0$ define $\psi_{k,p,q}\sceq \psi_{k,p,q}^{(h)}\colon \mc E\to V$ by
\[
\psi_{k,p,q}(z) \sceq \varphi_{p,q}(z)e_k
\]
Then $\{ \psi_{k,p,q} \colon 1\leq k\leq d,\ 1\leq p\leq N(h),\ q\in\N_0\}$ is a Hilbert basis for $H^2(\mc E;V)$.

In turn,
\begin{equation}\label{TOHSnorm}
 \left\| \TO_{s,\lambda}^N \right\|^2_{\HS} = \sum_{q\in\N_0} \sum_{p=1}^{N(h)} \sum_{k=1}^d \left\| \TO_{s,\lambda}^N \psi_{k,p,q}\right\|^2.
\end{equation}
In what follows, we evaluate step by step the right hand side of \eqref{TOHSnorm}, proceeding from the most inner norm to the final outer series.

Let $k\in\{1,\ldots, d\}$, $p\in\{1,\ldots, N(h)\}$, $q\in\N_0$. Then we have
\begin{align*}
\left\| \TO_{s,\lambda}^N \psi_{k,p,q} \right\|^2 & = \int_{\mc E} \left| \TO_{s,\lambda}^N \psi_{k,p,q}(z)\right|^2 \, \dvol(z)
\\
& = \sum_{j=1}^{2m} \int_{\mc E\cap \mc D_j} \sum_{\alpha,\beta\in \mc W_N^j} \left\langle \nu_s(\gamma_\alpha)\psi_{k,p,q}(z), \nu_s(\gamma_\beta)\psi_{k,p,q}(z) \right\rangle\, \dvol(z)
\\
& = \sum_{j=1}^{2m} \sum_{\alpha,\beta\in\mc W_N^j} \langle \lambda(\gamma_\alpha)e_k, \lambda(\gamma_\beta)e_k\rangle 
\\
& \quad\times \int_{\mc E\cap \mc D_j} \left( \big(\gamma_{\alpha}^{-1}\big)'(z) \right)^s \overline{\left( \big(\gamma_{\beta}^{-1}\big)'(z) \right)^s} \varphi_{p,q}\big(\gamma_{\alpha}^{-1}.z\big) \overline{\varphi_{p,q}\big(\gamma_\beta^{-1}.z\big)}  \,\dvol(z).
\end{align*}

Let 
\[
 \chi_\lambda(\gamma)\sceq \Tr\lambda(\gamma)
\]
for $\gamma\in\Gamma$. In order to evaluate the sum over $k$ in \eqref{TOHSnorm}, we note that 
\[
 \sum_{k=1}^d \left\langle \lambda\big(\gamma_\alpha\big)e_k, \lambda\big(\gamma_\beta\big)e_k\right\rangle = \Tr \lambda\big(\gamma_\alpha\gamma_\beta^{-1}\big) = \chi_\lambda\big(\gamma_\alpha\gamma_\beta^{-1}\big)
\]
for all $\alpha,\beta\in\mc W_N$. Thus
\begin{align}
\sum_{k=1}^d \left\| \TO_{s,\lambda}^N \psi_{k,p,q}\right\|^2 \nonumber
 & = \sum_{j=1}^{2m} \sum_{\alpha,\beta\in\mc W_N^j} \chi_\lambda\big(\gamma_\alpha\gamma_\beta^{-1}\big)  
 \\
 &\quad \times \int_{\mc E\cap \mc D_j} \left( \big(\gamma_{\alpha}^{-1}\big)'(z) \right)^s \overline{\left( \big(\gamma_{\beta}^{-1}\big)'(z) \right)^s} \varphi_{p,q}\big(\gamma_{\alpha}^{-1}.z\big) \overline{\varphi_{p,q}\big(\gamma_\beta^{-1}.z\big)}  \,\dvol(z). 
  \label{TOab}
\end{align}
Lemma~\ref{lem:vanishing} implies that in \eqref{TOab} only the summands with $\alpha=\beta$ contribute. Hence
\begin{align*}
  \sum_{k=1}^d \left\| \TO_{s,\lambda}^N \psi_{k,p,q}\right\|^2 & = d \sum_{j=1}^{2m} \sum_{\alpha\in\mc W_N^j} \int_{\mc E\cap \mc D_j} \left|\left( \big(\gamma_{\alpha}^{-1}\big)'(z) \right)^s\right|^2  \left|\varphi_{p,q}\big(\gamma_{\alpha}^{-1}.z\big)\right|^2 \, \dvol(z).
\end{align*}
For $j\in\{1,\ldots, 2m\}$, $\alpha\in \mc W_N^j$, $u\in \mc P_j$ let $v=v(u,\alpha) \in \{1,\ldots, N(h)\}$ be the unique element such that $\gamma_\alpha^{-1}(\mc E_u)\subseteq \mc E_v$. Then 
\begin{align*}
 \sum_{p=1}^{N(h)}\sum_{k=1}^d \left\| \TO_{s,\lambda}^N \psi_{k,p,q}\right\|^2 & =
 d\sum_{j=1}^{2m} \sum_{\alpha\in\mc W_N^j} \sum_{u\in\mc P_j} \sum_{p=1}^{N(h)} \int_{\mc E_u} \left|\left( \big(\gamma_{\alpha}^{-1}\big)'(z) \right)^s\right|^2  \left|\varphi_{p,q}\big(\gamma_{\alpha}^{-1}.z\big)\right|^2 \,\dvol(z)
 \\
 & = d \sum_{j=1}^{2m} \sum_{\alpha\in\mc W_N^j} \sum_{u\in\mc P_j} \int_{\mc E_u} \left|\left( \big(\gamma_{\alpha}^{-1}\big)'(z) \right)^s\right|^2  \left|\varphi_{v(u,\alpha),q}\big(\gamma_{\alpha}^{-1}.z\big)\right|^2 \,\dvol(z).
\end{align*}
To evaluate the final outer series in \eqref{TOHSnorm} we use that, for each $v\in\{1,\ldots, N(h)\}$, the series 
\[
 \sum_{q\in\N_0} \varphi_{v,q}\overline\varphi_{v,q}
\]
converges compactly to the Bergman kernel $B_{\mc E_v}$ of $\mc E_v$, and that due to the specific shape of $\mc E_v$ (a complex ball), the Bergman kernel $B_{\mc E_v}$ is given by a rather easy explicit formula (which in this case also follows from a straightforward calculation).  More precisely, for all $(z,w)\in \mc E\times\mc E$ we have
\begin{equation}\label{bergmankernel}
 \sum_{q\in\N_0} \varphi_{v,q}(z)\overline{\varphi_{v,q}(w)} = B_{\mc E_v}(z,w) = \frac{r_v^2}{\pi\left( r_v^2 - (w-c_v)(\overline{z}-c_v)\right)^2}.
\end{equation}
Hence
\begin{align}
 \left\| \TO_{s,\lambda}^N \right\|^2_{\HS} & =
 \sum_{q\in\N_0} \sum_{p=1}^{N(h)} \sum_{k=1}^d \left\| \TO_{s,\lambda}^N \psi_{k,p,q}\right\|^2 \nonumber
 \\
 & = d \sum_{j=1}^{2m} \sum_{\alpha\in\mc W_N^j} \sum_{u\in\mc P_j} \int_{\mc E_u} \left|\left( \big(\gamma_{\alpha}^{-1}\big)'(z) \right)^s\right|^2  B_{\mc E_{v(u,\alpha)}}\left(\gamma_\alpha^{-1}.z, \gamma_\alpha^{-1}.z\right) \,\dvol(z). \label{TO_berg}
\end{align}
For all $j\in \{1,\ldots, 2m\}$, all $\alpha\in\mc W_N^j$, $u\in\mc P_j$, $z\in \mc E_u$, the combination of \eqref{bergmankernel} with \eqref{contained} yields that 
\begin{equation}\label{est_bergman}
 \left| B_{\mc E_{v(u,\alpha)}}\left(\gamma_\alpha^{-1}.z, \gamma_\alpha^{-1}.z\right)\right| \leq c_1 h^{-2}
\end{equation}
where
\[
 c_1 \sceq \frac{16}{\pi}C
\]
depends on $X$ only. 

From now let $\sigma_1>\sigma_0\geq 0$ and $T_0\in\RR$ be fixed, and set
\[
 D\sceq (\sigma_0,\sigma_1) + i(T-T_0, T+T_0).
\]
By \cite[below (11)]{Naud_inventiones} there exists $c_2 = c_2(\sigma_0,\sigma_1, T_0)>0$ such that for all $j\in\{1,\ldots, 2m\}$ and all $s\in D$ we have
\begin{equation}\label{est_intermediate}
 \sup\left\{ \left| \left( \left(\gamma_\alpha^{-1}\right)'(z)\right)^s \right| : \alpha\in \mc W_N^j,\ z\in \mc E_u,\ \mc E_u\subseteq \mc D_j \right\} \leq c_2 \sup_{x\in I_j} \left( \left(\gamma_\alpha^{-1}\right)'(x) \right)^{\Rea s}.
\end{equation}
In \cite{Naud_inventiones} this estimate is shown for the case that $|\Ima s|=h^{-1}$. In this case, the constant $c_2$ is independent of $h$ (and hence of $T$). Any continuous perturbation of $T$ then results in a continuous perturbation of $c_2$. Thus, applied to all $s\in D$, the constant $c_2$ remains independent of $T$ but depends (continuously) on $T_0$.

Using \eqref{est_bergman} and \eqref{est_intermediate} in \eqref{TO_berg} we get for all $s\in D$,
\begin{equation}\label{TO_est}
 \left\| \TO_{s,\lambda}^N \right\|^2_{\HS} \leq c_3 d h^{-2} \sum_{j=1}^{2m} \sum_{\alpha\in\mc W_N^j} \sup_{x\in I_j} \left( \left(\gamma_\alpha^{-1}\right)'(x) \right)^{ 2 \Rea s} \sum_{u\in\mc P_j}\int_{\mc E_u} 1 \, \dvol(z)
\end{equation}
for a constant $c_3>0$ with the same dependencies as $c_2$. From \cite{Sullivan} (see also \cite[Section~5]{Guillope_Lin_Zworski} or \cite[Proof of Theorem~15.12]{Borthwick_book}) it follows that 
\[
 \#\mc P_j \leq N(h) \leq c_4 h^{-\delta}
\]
for some constant $c_4>0$ depending on $X$ only. Further, for any $u\in \{1,\ldots, N(h)\}$, 
\[
 \int_{\mc E_u} 1 \, \dvol(z) \leq \pi \left( \frac{Ch}{2}\right)^{2}.
\]
By \cite[Lemma~3.1]{Naud_inventiones} there exists a map $p\colon\RR \to \RR$ that is strictly convex, strictly decreasing, and has a unique zero which is precisely $\delta$ and a constant $c_5=c_5(\sigma_0,\sigma_1)$ such that for all $s\in\RR$ with $\Rea s\in (\sigma_0,\sigma_1)$ we have
\begin{equation}\label{est_global}
 \sum_{j=1}^{2m} \sum_{\alpha\in\mc W_N^j} \sup_{x\in I_j} \left( \left(\gamma_\alpha^{-1}\right)'(x) \right)^{\Rea s} \leq c_5 e^{Np(\sigma_0)}.
\end{equation}
The function $p$ is a rescaled variant of the topological pressure of the discrete dynamical system that gives rise to the transfer operator $\TO_s$. We refer to \cite{Naud_inventiones} for more details.

Using these estimates in \eqref{TO_est} we get for all $s\in\CC$, $\Rea s\in (\sigma_0,\sigma_1)$, 
\[
 \left\| \TO_{s,\lambda}^N \right\|^2_{\HS} \leq c_6 d h^{-\delta} e^{Np(2\sigma_0)}
\]
where $c_6$ depends on $\sigma_0,\sigma_1, T_0$ and $X$ only, and the dependence on $T_0$ is continuous. 

Inserting the values for $h$ and $N$ as defined in the beginning of this proof, using $d=\dcov(\wt X,X)$, and combining with \eqref{HSest} completes the proof.
\end{proof}

\subsection{Proof of Theorem~\ref{thm:box}}\label{sec:finalproof}
If $X$ is an elementary Schottky surface, hence a hyperbolic cylinder, then the statement of Theorem~\ref{thm:box} is obviously true and the obtained bound on the number of resonances is sharp because for each finite cover $\wt X$ of $X$, both $M_{\wt X}(\sigma, T)$ and $\zvol(\wt X)$ vanish.   

Suppose that $X$ is non-elementary and let $\sigma>\delta/2$. Further let $\wt X$ be a finite cover of $X$ and let $T\in\RR$. The value of $M_{\wt X}(\sigma,T)$ that we seek to estimate is the number of resonances in the rectangle 
\[
 R(\sigma,T) \sceq [\sigma, \delta] + i[T-1,T+1].
\]
We use Titchmarsh's Number of Zeros Theorem to provide such an estimate. To that end let
\[
 z_0\sceq 2 + iT
\]
and fix $r_2>r_1>0$ such that 
\[
 R(\sigma, T)\subseteq \overline B(z_0;r_1)
\]
and $2-r_2 > \delta/2$, thus
\[
 \overline B(z_0;r_2) \subseteq \{ z\in\CC : \Rea z> \delta/2\}.
\]
Note that the choice of $r_1,r_2$ may depend on $\sigma$ but it is independent of $T$.

Further let $N(T,\wt X)$ be as in Proposition~\ref{prop:NTX} and set $f\colon \CC\to \CC$, 
\[
 f(s) \sceq \det\left( 1 - \TO_{s,\lambda}^{2 N(T,\wt X)}\right).
\]
Then 
\begin{align*}
 M_{\wt X}(\sigma,T) \leq \#\{ s\in \mc R(\wt X) : s\in \overline B(z_0;r_1),\ f(s) = 0 \}.
\end{align*}
Titchmarsh's Number of Zeros Theorem yields
\begin{align}\label{box_est}
 M_{\wt X}(\sigma,T) \leq \frac{1}{\log(r_2/r_1)} \left( \log \max_{|s-z_0|=r_2} |f(s)| - \log |f(z_0)| \right).
\end{align}
Since $\Rea z_0=2>\delta$ we can use Proposition~\ref{last_bound} to estimate the second summand in \eqref{box_est}. To estimate the first summand, we use Proposition~\ref{prop:NTX} with $\sigma_0\sceq 2-r_2$, $\sigma_1\sceq 2+r_2$, $T_0=r_2$. Thus, there is a function $\eta\colon\RR\to\RR$ with the properties as stated in Proposition~\ref{prop:NTX} and constants $c_1,c_2>0$ depending on $\sigma$ (and $X$) only such that 
\begin{align*}
M_{\wt X}(\sigma,T) \leq c_1 \zvol(\wt X) e^{-\eta(\sigma_0)\ell_0(\wt X)} \langle T\rangle^{\delta-\eta(\sigma_0)} + c_2 \zvol(\wt X) N(T,\wt X) e^{-(2-\delta)\ell_0(\wt X)}.
\end{align*}
Recall from Proposition~\ref{prop:NTX} that 
\[
 N(T,\wt X) \approx c_3\ell_0(\wt X) + c_4 \log\langle T\rangle + c_5
\]
for certain constants $c_3,c_4,c_5>0$ depending on $X$ only. Since 
\[
0\leq \log\langle T\rangle \leq c_\eps \langle T\rangle^\eps
\]
for all $\eps>0$, and 
\[
 \ell_0(\wt X) e^{-\frac12\ell_0(\wt X)}
\]
is bounded as $\ell_0(\wt X)\to\infty$, we find $\tau_1(\sigma)>0$, $\tau_2(\sigma)\in (0,\delta)$, $c>0$ depending on $\sigma$ and $X$ only such that 
\[
 M_{\wt X}(\sigma, T) \leq c \zvol(\wt X) e^{-\tau_1(\sigma)\ell_0(\wt X)} \langle T\rangle^{\delta-\tau_2(\sigma)}.
\]
Due to the properties of $\eta$, the functions $\tau_j\colon\sigma\to\tau_j(\sigma)$ ($j=1,2$) can be chosen as stated in Theorem~\ref{thm:box}. This completes the proof of Theorem~\ref{thm:box}.

\section{Examples for covers with long shortest geodesics}\label{sec:eigen}

In this section we provide additional examples of sequences of finite covers of integral Schottky surfaces along which an analogue of \eqref{JNgrowth} holds. The key tool for these results is Theorem~\ref{thm:box} in combination with a study of the length of minimal periodic geodesics in relation to the covering degree. More precisely, we seek to find towers of coverings $ (X_j) $ whose minimal lengths $ \ell_0(X_j) $ grow logarithmically in $ \dcov(X_j, X) $ as $ j\to \infty $. 

We note that along sequences of abelian covers such growth behavior is impossible, see \cite{JNS}. For completeness, we recall that a cover $\wt X = \wt\Gamma\backslash\hh$ of $X=\Gamma\backslash\hh$ is called \textit{abelian} if $\wt\Gamma$ is a finite normal subgroup of $\Gamma$, and the quotient group $\Gamma/\wt\Gamma$ is abelian. Hence, when trying to produce covers with large minimal geodesic length, this suggests that we should look for highly non-abelian covers. As it turns out, families of `congruence' covers are good candidates. 

Let $X=\Gamma\backslash\hh$ be an integral Schottky surface, thus, $\Gamma\subseteq \SL_2(\Z)$. We consider $X$ and $\Gamma$ to be fixed throughout this section. For $q\in\N$ let (allowing a slight abuse of notation for notational convenience)
\begin{align*}
 \Gamma_0(q) &\sceq \left\{ g\in\Gamma : g \equiv \textmat{*}{*}{0}{*} \mod q \right\},
 \\
 \Gamma_1(q) &\sceq \left\{ g\in\Gamma : g \equiv \textmat{1}{*}{0}{1} \mod q \right\},
 \\
 \Gamma_2(q) &\sceq \left\{ g\in\Gamma : g \equiv \textmat{1}{0}{0}{1} \mod q \right\} = \Gamma(q).
\end{align*}
For $j\in\{0,1,2\}$, we set 
\[
 X_j(q) \sceq \Gamma_j(q)\backslash\hh.
\]
As shown in Proposition~\ref{prop:eigengrowth} below, along any sequence $(X_q)_{q\in\N}$ of covers of $X$ sandwiched between $(X_0(q))_{q\in\N}$ and $(X_2(q))_{q\in\N}$, an analogue of~\eqref{JNgrowth} holds. The sequence $(X_1(q))_{q\in\N}$ is one such example. For the sequence $(X_2(q))_q$, Proposition~\ref{prop:eigengrowth} recovers the result by Jakobson--Naud. 

\begin{prop}\label{prop:eigengrowth}
For each $q\in\N$ let $\Gamma_q$ be a Schottky group such that $\Gamma_2(q)\subseteq \Gamma_q\subseteq \Gamma_0(q)$ and set $X_q\sceq \Gamma_q\backslash\hh$. Then there exists functions $\alpha, \beta\colon\RR\to\RR$ that are strictly concave, increasing, and positive on $(\delta/2,\delta]$ such that for each $\sigma>\delta/2$ there exists $C>0$ such that for all $T\geq 1$ and all $ q\in\N $ (not necessarily prime), we have
\begin{equation}
 M_{X_q}(\sigma, T) \leq C [\Gamma : \Gamma_q]^{1-\alpha(\sigma)} \langle T\rangle^{\delta-\beta(\sigma)}.
\end{equation}
In particular, there exists $\alpha>0$ such that the number of $ L^{2} $-eigenvalues satisfies
\[
 \#\Omega(X_q) = O\big([\Gamma:\Gamma_q]^{1-\alpha}\big)\qquad\text{as $q\to\infty$.}
\]
\end{prop}

Proposition~\ref{prop:eigengrowth} follows immediately from Proposition~\ref{prop:actualproof} below in combination with Theorem~\ref{thm:box}. Before we discuss Proposition~\ref{prop:actualproof}, we present the following lemma on the growth of the covers.

\begin{lem}
For $j\in\{0,1,2\}$ we have
\[
 [\Gamma : \Gamma_j(q)] \asymp q^{j+1} \quad\text{as $q\to\infty$, $q$ prime.}
\]
\end{lem}

\begin{proof}
Let $q\in\N$ be prime. Let 
\[
 \pi_q \colon \Gamma\to \SL_2(\Z/q\Z),\quad g\mapsto g\mod q.
\]
For $j\in\{0,1,2\}$ we let $H_j(q)$ denote the subgroup of $\SL_2(\Z/q\Z)$ given by
\[
 H_0\sceq \left\{ \begin{pmatrix} * & * \\ 0 & * \end{pmatrix} \right\},\quad H_1 \sceq \left\{ \begin{pmatrix} 1 & * \\ 0 & 1 \end{pmatrix} \right\},\quad H_2 \sceq \{\id\}.
\]
Then 
\[
 \Gamma_j(q) \sceq \pi_q^{-1}\big(H_j(q)\big),\quad (j\in\{0,1,2\}).
\]
By \cite[Section~2]{Gamburd}, the map $\pi_q$ is surjective if $q$ is sufficiently large. Thus, the isomorphism theorems for groups show that for all such sufficiently large $q$ and each $j\in\{0,1,2\}$ we have
\begin{equation}\label{index_transform}
 [\Gamma : \Gamma_j(q)] = [\SL_2(\Z/q\Z) : H_j(q)] = \frac{|\SL_2(\Z/q\Z)|}{|H_j(q)|}.
\end{equation}
As is well-known, $|\SL_2(\Z/q\Z)| = q(q^2-1)$. Obviously, $|H_2(q)|=1$ and $|H_1(q)|=q$. Since $q$ is prime, $\Z/q\Z$ is a field and hence contains $q-1$ multiplicatively invertible elements. Thus, there are $q-1$ possibilities for the pair of diagonal entries of an element of $H_0$. Hence, $|H_0|=q(q-1)$. Using these element counts in \eqref{index_transform} completes the proof.
\end{proof}

\begin{prop}\label{prop:actualproof}
Under the hypotheses of Proposition~\ref{prop:eigengrowth} there exists $c_0>0$ such that for all $q\in\N$ we have
\[
 \ell_0\big(X_q\big) \geq c_0 \log [\Gamma : \Gamma_q].
\]
\end{prop}

\begin{proof}
For any $q\in\N$ we have 
\[
[\Gamma : \Gamma_0(q)] \leq [\Gamma : \Gamma_q]\leq  [\Gamma : \Gamma_2(q)] \leq |\SL_2(\Z/q\Z)| = q^3 \prod_{\substack{\text{$p$ prime}\\ \text{divisor of $q$}}} \left(1 - \frac{1}{p^2}\right) < q^3.
\]
Thus, it suffices to establish the existence of $c_0>0$ such that 
\begin{equation}\label{c_bound}
 \ell_0\big(X_q\big) \geq c_0 \log q
\end{equation}
for all $q\in\N$. Since, for each $q\in\N$, the group $\Gamma_q$ is contained in $\Gamma_0(q)$, the shortest geodesic on $X_q$ is at least as long as the shortest geodesic on $X_0(q)$. Hence, to establish \eqref{c_bound} it suffices to prove the existence of $c_0>0$ such that for all $q\in\N$ we have
\begin{equation}\label{c_bound2}
 \ell_0\big(X_0(q)\big) \geq c_0 \log q.
\end{equation}
To that end let $q\in\N$ and let 
\[
 g = \begin{pmatrix} a & b \\ c & d \end{pmatrix} \in \Gamma_0(q)
\]
be hyperbolic. Since, necessarily, $|b|\geq 1$ and $|c|\geq q$, it follows that 
\[
 |\Tr g| = |a+d|\geq \frac12 \sqrt{|ad|} = \frac12\sqrt{|1+bc|} \geq \frac12 \sqrt{q-1} \geq \frac{1}{2\sqrt{2}}q^\frac12.
\]
Thus,
\[
 \ell(g) \geq 2\log \frac{|\Tr g|}{2} \geq  \log\frac{q}{4\sqrt{2}}.
\]
In turn,
\[
 \ell_0\big(X_0(q)\big) \geq \log\frac{q}{4\sqrt{2}}. 
\]
We find $c_0>0$ such that for all $q\geq 6$, 
\[
 \log \frac{q}{4\sqrt{2}} \geq c_0 \log q,
\]
which shows \eqref{c_bound2} for $q\geq 6$. By shrinking $c_0>0$ sufficiently (if necessary) we can establish \eqref{c_bound2} for $q\in\{1,2,\ldots, 5\}$ as well. This establishes \eqref{c_bound} and hence completes the proof.
\end{proof}

\section{Regular covers and Cayley graphs}\label{sec:cayley}

Throughout this section let $X=\Gamma\backslash\hh$ be a fixed non-elementary Schottky surface, let $S$ be a fixed set of generators for $\Gamma$ and suppose that $S$ is symmetric (i.\,e., $S^{-1}=S$). For convenience, we suppose that $S=\{\gamma_1,\ldots, \gamma_{2m}\}$ is the set of generators arising from a geometric construction of $\Gamma$, see Section~\ref{sec:Schottky}.

Let $\wt X = \wt\Gamma\backslash\hh$ be a finite \textit{regular} cover of $X$, that is, $\wt\Gamma$ is normal in $\Gamma$. Let $\bG\sceq \Gamma/\widetilde{\Gamma}$ be the quotient group, and let $\pi \colon \Gamma \to \bG$ be the natural projection. We associate to the pair $(X,\wt X)$ the Cayley graph 
\[
 \mathcal{G} \sceq \mathrm{Cay}\left( \bG, \pi(S) \right)
\]
of $\bG$ with respect to $ \pi(S) $. Note that $ \pi(S) $ is a symmetric generating set for $\bG$, and hence $ \mathcal{G} $ is a simple, connected graph. Recall that the girth of $ \mathcal{G} $, denoted by  $\mathrm{girth}(\mathcal{G})$, is the length of the shortest cycle in $ \mathcal{G}$ or, equivalently, the length of the shortest non-trivial relation in the group $ \bG $ with respect to the generating set $ \pi(S) $.

In this section we show the following bound of the resonance counting function $M_{\wt X}$, which is essentially a corollary of Theorem~\ref{thm:box} and provides a rather algebraic interpretation of it.

\begin{cor}\label{Covers_Cayley}
Let $ \widetilde{X}\to X $ be a finite regular cover and let $ \mathcal{G} $ be the associated Cayley graph. Then for all $ \sigma>\delta/2 $ and $ T\in \mathbb{R} $ we have
$$ M_{\widetilde{X}}(\sigma, T) \leq C  \vert \G\vert e^{-\alpha_{1} \mathrm{girth}(\mathcal{G})}\langle T\rangle^{\delta - \alpha_{2}}, $$ 
for constants $ C, \alpha_{1}, \alpha_{2} >0 $ depending solely on $ \sigma $ and $ X $.
\end{cor}

\begin{rem}
\begin{enumerate}[{\rm (i)}]
\item If $ \G $ is a abelian (i.e. the covering $ \widetilde{X}\to X $ is abelian), then $ \mathrm{girth}(\mathcal{G}) \leq 4 $, since $ \mathcal{G} $ must contain a $ 4 $-cycle ($ a+b-a-b=0 $). By contrast, Cayley graphs $ \mathcal{G} $ of the group $ \G = \mathrm{SL}_{2}(\mathbb{Z}/q\mathbb{Z}) $ with $ q $ prime have logarithmic girth, i.e. $ \mathrm{girth}(\mathcal{G})\gg \log \vert \G\vert $, see \cite[Section~2]{Gamburd}.
\item Let $ (X_{n})_{n} $ be a sequence of finite covers of a Schottky surface $ X $ and let $ (\mathcal{G}_{n})_{n} $ be the associated sequence of Cayley graphs. Corollary \ref{Covers_Cayley} shows that if $ \mathrm{girth}(\mathcal{G}_{n})\to \infty $, then the number of resonances satisfy  
$$
\frac{M_{X_{n}}(\sigma, T)}{\vert \G_{n}\vert} \to 0 \quad \text{as} \quad n\to \infty,
$$
for fixed $ \sigma > \delta/2 $ and $ T\in \mathbb{R} $.
\end{enumerate}
\end{rem}

For $\gamma\in\Gamma$ let $ L_{S}(\gamma) $ denote the minimal word length of $\gamma$ over the alphabet $S$ (i.\,e., the representing word of $\gamma$ does not contain neighboring pairs of mutually inverses). Further, let 
$$ \mathrm{WL}(\gamma) := \min_{g\in \overline{\gamma}} L_{S}(g) $$
be the shortest word length of a representative in the conjugacy class of $ \gamma $. 

A crucial observation for the proof of Corollary~\ref{Covers_Cayley} is that for each $\gamma\in\Gamma\setminus\{\id\}$, $\mathrm{WL}(\gamma) $ is bounded by the (displacement) length $\ell(\gamma)$.

\begin{lem}\label{word_vs_hyperbolic_length}
There exists a constant $ C > 0 $ (depending only on $ \Gamma $) such that for all $ \gamma \in \Gamma\setminus\{\id\}$ we have
$$ \mathrm{WL}(\gamma) \leq  C \cdot \ell(\gamma).$$ 
\end{lem}

\begin{proof}
Let $ \overline{\mathcal{N}} \subset \mathbb{H} $ be the \emph{Nielsen region} of $ X = \Gamma\setminus \mathbb{H} $, that is,  the union of all geodesic arcs connecting two points in the limit set $\Lambda(\Gamma)$ of $\Gamma$. Let $ \mathcal{N}:= \Gamma\setminus \overline{\mathcal{N}} $ denote the \textit{convex core} of $ X $. Since $ X $ is a Schottky surface and hence convex co-compact, $ \mathcal{N} $ is compact. Let $\wt{\mathcal N}$ be a compact subset of $\overline{\mathcal N}$ that contains at least one representative for each point in $\mathcal N$. Let $d_\hh$ denote the hyperbolic metric on $\hh$. 

By Knopp--Sheingorn \cite{Knopp_Sheingorn} we find constants $ c_{1}, c_{2}>0 $ such that for every $ \gamma\in \Gamma \setminus \{ \id \} $ we have
$$
L_{S}(\gamma) \leq c_{1} d_\hh(\gamma i, i) + c_{2}.
$$
Since $\wt{\mathcal{N}} $ is compact we find $ c_{3}>0$ such that for all $ z'\in \wt{\mathcal{N}} $ we have $ d_\hh(z',i)\leq c_{3} $. Now let $ z\in \overline{\mathcal{N}} $ be arbitrary. Clearly, there exists $ h\in \Gamma $ such that $ z':= h z\in \wt{\mathcal{N}}. $ Note that
$$
\mathrm{WL}(\gamma)\leq L_{S}(h\gamma h^{-1}) \leq c_{1}\cdot d_\hh(h\gamma h^{-1} i, i) + c_{2}
$$
Using the triangle inequality and exploiting left-invariance of $d_\hh$ leads to
\begin{align*}
d_\hh(h\gamma h^{-1} i, i) &\leq d_\hh(h\gamma h^{-1} i, h\gamma  z) + d_\hh(h\gamma  z, h z) + d_\hh(h z, i)\\
&= d_\hh(h^{-1} i,  z) + d_\hh(\gamma  z, z) + d_\hh(z', i)\\
&= d_\hh(i, z') + d_\hh(\gamma  z, z) + d_\hh(z', i)\\
&= d_\hh(\gamma  z,  z) + 2 d_\hh(z', i)\\
&\leq d_\hh(\gamma  z,  z) + 2 c_{3}.
\end{align*}
Thus, there exists $c_4>0$ such that for every $ z\in \overline{\mathcal{N}} $ and every $\gamma\in\Gamma\setminus\{\id\}$ we have
$$
\mathrm{WL}(\gamma)\leq c_{1} d_\hh(\gamma z, z) + c_{4}.
$$
For each $\gamma\in\Gamma\setminus\{\id\}$, there is an element, say $\eta$, in the conjugacy class $\overline\gamma$ such that the geodesic on $\hh$ connecting the two fixed points of $\eta$ passes through $\wt{\mathcal N}$. Let $ \alpha(\gamma) $ be the geodesic arc connecting the two fixed points of $ \gamma $. Thus, there is $ z\in \alpha(\gamma)\subset \overline{\mathcal{N}} $ such that $ \ell(\gamma) = \ell(\eta) = d_\hh(\eta z, z). $
Since $ \ell(\gamma)=\ell(\eta) $ is bounded from below by $ \ell_{0}(X) > 0 $, we obtain 
$$ \mathrm{WL}(\gamma)\leq c_{1} \ell(\gamma) + c_{4} \leq c_{5} \ell(\gamma) $$ 
for a constant $ c_{5}>0 $ depending on $\Gamma$ only.
\end{proof}

\begin{proof}[Proof of Corollary~\ref{Covers_Cayley}]
In view of Theorem \ref{thm:box} it suffices to show that $ \ell_{0}(\widetilde{X})\geq c \cdot\mathrm{girth}(\mathcal{G}) $ for some   $ c > 0 $ only depending on $ \Gamma. $

Clearly, every element $ \gamma\in \wt\Gamma \setminus \{ \mathrm{id}\} $ can be written as a reduced word $ \gamma_{i_{1}}\gamma_{i_{2}}\cdots \gamma_{i_{L}} $ with $ L=L_{S}(\gamma) > 0 $ and indices $ i_{1}, \dots, i_{L}\in \{ 1, \dots,2m\}. $

Pick an element $ \gamma\in \widetilde{\Gamma}\setminus \{ \mathrm{id}\} $ with \textit{minimal} word length. By the assumption of minimality, we can write $ \gamma $ as a reduced word $ \gamma_{i_{1}}\gamma_{i_{2}}\cdots \gamma_{i_{L}} $ with $ L = \mathrm{WL}(\gamma). $  Set $ g_{i} :=\pi(\gamma_{i})\in \G $ for each $ i\in \{ 1, \dots,2m\} $. Clearly, since $ \gamma\in \widetilde{\Gamma}\setminus \{ \mathrm{id}\} $ we have 
$$
\mathrm{id}_{\G} = \pi(\gamma) =  g_{i_{1}}g_{i_{2}}\cdots g_{i_{L}}.
$$ 
For $ j=1,\dots, L $ set $ x_{j}:= g_{i_{1}}g_{i_{2}}\cdots g_{i_{j}}\in \G $. Using again the assumption of minimality of $ L $, it is easy to see that the elements $ x_{1}, \dots, x_{L} = \mathrm{id}_{\G} $ are all distinct. This yields the cycle 
$$
\mathrm{id}_{\G} \to x_{1} \to \cdots \to x_{L}= \mathrm{id}_{\G}
$$
in $ \mathcal{G} $ of length $ L . $ Since the girth of $ \mathcal{G} $ is by definition the length of the shortest cycle, it follows that
\begin{equation}\label{minimality}
\min_{\gamma\in \widetilde{\Gamma}\setminus \{ \mathrm{id}\}} \mathrm{WL}(\gamma) \geq \mathrm{girth}(\mathcal{G}).
\end{equation}
By Lemma \ref{word_vs_hyperbolic_length}, we have $ \ell(\gamma) \geq C^{-1} \mathrm{WL}(\gamma) $, which combined with \eqref{minimality} yields 
$$
\ell_{0}(\wt X) = \min_{\gamma\in \widetilde{\Gamma}\setminus \{ \mathrm{id}\}} \ell(\gamma)\geq C^{-1}\mathrm{girth}(\mathcal{G}). 
$$
The proof of Corollary \ref{Covers_Cayley} is complete.
\end{proof}


\end{document}

%% file: makros.tex
\DeclareMathOperator{\dcov}{dcov}
\DeclareMathOperator{\Ind}{Ind}
\DeclareMathOperator{\id}{id}
\DeclareMathOperator{\Tr}{Tr}

\DeclareMathOperator{\zvol}{0-vol}
\DeclareMathOperator{\supp}{supp}
\DeclareMathOperator{\Ima}{Im}
\DeclareMathOperator{\Rea}{Re}
\DeclareMathOperator{\SL}{SL}
\DeclareMathOperator{\diam}{diam}
\DeclareMathOperator{\dvol}{dvol}
\DeclareMathOperator{\rank}{rank}
\DeclareMathOperator{\vol}{vol}
\DeclareMathOperator{\PSL}{PSL}

\DeclareMathOperator{\Res}{Res}

\newcommand{\trivrep}{{\bf 1}}
\newcommand{\TO}{\mathcal{L}}
\newcommand{\wh}{\widehat}
\newcommand{\HS}{\mathrm{HS}}

\newcommand{\eps}{\varepsilon}

\newcommand{\textbmat}[4]{\left[\begin{smallmatrix} #1&#2 \\ #3&#4
\end{smallmatrix}\right]}
\newcommand{\textmat}[4]{\left(\begin{smallmatrix} #1&#2 \\ #3&#4
\end{smallmatrix}\right)}

\newcommand{\bG}{\mathbf{G}}

\renewcommand{\Im}{{\rm Im}}
\newcommand{\RR}{\mathbb{R}}
\newcommand{\CC}{\mathbb{C}}
\newcommand{\Z}{\mathbb{Z}}
\newcommand{\N}{\mathbb{N}}

\newcommand{\hh}{\mathbb{H}}

\newcommand{\G}{{\bf G}}

\newcommand{\lt}{{\mathcal L}}

\newcommand{\wt}{\widetilde}
\newcommand{\mc}[1]{\mathcal #1}
\newcommand{\sceq}{\mathrel{\mathop:}=}

%% file: Density_Resonances.bbl
\begin{thebibliography}{10}

\bibitem{BMM}
W.~Ballmann, H.~Matthiesen, and S.~Mondal.
\newblock Small eigenvalues of surfaces of finite type.
\newblock {\em Compos. Math.}, 153(8):1747--1768, 2017.

\bibitem{sharpbounds}
D.~Borthwick.
\newblock Sharp geometric upper bounds on resonances for surfaces with
  hyperbolic ends.
\newblock {\em Anal. PDE}, 5(3):513--552, 2012.

\bibitem{Borthwick_book}
D.~{Borthwick}.
\newblock {Spectral theory of infinite-area hyperbolic surfaces. 2nd edition}.
\newblock {\em {Prog. Math.}}, 318, 2016.

\bibitem{BJP}
D.~Borthwick, C.~Judge, and P.~Perry.
\newblock Selberg's zeta function and the spectral geometry of geometrically
  finite hyperbolic surfaces.
\newblock {\em Comment. Math. Helv.}, 80(3):483--515, 2005.

\bibitem{BGS}
J.~Bourgain, A.~Gamburd, and P.~Sarnak.
\newblock Generalization of {Selberg's} 3/16 theorem and affine sieve.
\newblock {\em Acta Math.}, 207(2):255--290, 2011.

\bibitem{Bourgain_Kontorovich}
J.~Bourgain and A.~Kontorovich.
\newblock {On {Z}aremba's conjecture}.
\newblock {\em {Ann. Math. (2)}}, 180(1):1--60, 2014.

\bibitem{BDW}
S.~{Dyatlov}.
\newblock {Improved fractal Weyl bounds for hyperbolic manifolds}.
\newblock {\em {J. Eur. Math. Soc.}}, 21(6):1595--1639, 2019.

\bibitem{FP_szf}
K.~Fedosova and A.~Pohl.
\newblock Meromorphic continuation of {S}elberg zeta functions with twists
  having non-expanding cusp monodromy.
\newblock {\em Selecta Math. (N.S.)}, 26(1):Paper No. 9, 2020.

\bibitem{Gamburd}
A.~Gamburd.
\newblock On the spectral gap for infinite index ``congruence'' subgroups of $
  \mathrm{SL}_{2}(\mathbb{R}) $.
\newblock {\em Israel J. Math.}, 127:157--200, 2002.

\bibitem{Guillope_Lin_Zworski}
L.~{Guillop\'e}, K.~{Lin}, and M.~{Zworski}.
\newblock {The Selberg zeta function for convex co-compact Schottky groups}.
\newblock {\em {Commun. Math. Phys.}}, 245(1):149--176, 2004.

\bibitem{GZ_upper_bounds}
L.~Guillop\'{e} and M.~Zworski.
\newblock Upper bounds on the number of resonances for non-compact {R}iemann
  surfaces.
\newblock {\em J. Funct. Anal.}, 129(2):364--389, 1995.

\bibitem{GZ_scattering_asympt}
L.~{Guillop\'e} and M.~{Zworski}.
\newblock {Scattering asymptotics for Riemann surfaces}.
\newblock {\em {Ann. Math. (2)}}, 145(3):597--660, 1997.

\bibitem{GZ_Wave}
L.~Guillop\'{e} and M.~Zworski.
\newblock The wave trace for {R}iemann surfaces.
\newblock {\em Geom. Funct. Anal.}, 9(6):1156--1168, 1999.

\bibitem{JN}
D.~Jakobson and F.~Naud.
\newblock Resonances and density bounds for convex co-compact congruence
  subgroups of $ \mathrm{SL}_{2}(\mathbb{Z}) $.
\newblock {\em Israel J. Math.}, 213(1):443--473, 2016.

\bibitem{JNS}
D.~Jakobson, F.~Naud, and L.~Soares.
\newblock Large covers and sharp resonances of hyperbolic surfaces.
\newblock To appear in Ann. Institut Fourier.

\bibitem{Knopp_Sheingorn}
M.~{Knopp} and M.~{Sheingorn}.
\newblock {On Dirichlet series and Hecke triangle groups of infinite volume}.
\newblock {\em {Acta Arith.}}, 76(3):227--244, 1996.

\bibitem{Lalley}
S.~{Lalley}.
\newblock {Renewal theorems in symbolic dynamics, with applications to geodesic
  flows, noneuclidean tessellations and their fractal limits}.
\newblock {\em {Acta Math.}}, 163(1-2):1--55, 1989.

\bibitem{Lax_Phillips_I}
P.~Lax and R.~S. Phillips.
\newblock Translation representation for automorphic solutions of the wave
  equation in non-{Euclidean} spaces, {I}.
\newblock {\em Commun. Pure Appl. Math.}, 37(3):303--328, 1984.

\bibitem{Lu_Sridhar_Zworski}
W.~Lu, S.~Sridhar, and M.~Zworski.
\newblock Fractal {W}eyl laws for chaotic open systems.
\newblock {\em Phys. Rev. Lett.}, 91:154101, 2003.

\bibitem{Mazzeo_Melrose}
R.~Mazzeo and R.~Melrose.
\newblock Meromorphic extension of the resolvent on complete spaces with
  asymptotically constant negative curvature.
\newblock {\em J. Funct. Anal.}, 75:260--301, 1987.

\bibitem{Mueller_scattering}
W.~{M\"uller}.
\newblock {Spectral geometry and scattering theory for certain complete
  surfaces of finite volume}.
\newblock {\em {Invent. Math.}}, 109(2):265--305, 1992.

\bibitem{Naud_inventiones}
F.~Naud.
\newblock Density and location of resonances for convex co-compact hyperbolic
  surfaces.
\newblock {\em Invent. Math.}, 195(3):723--750, 2014.

\bibitem{NPS}
F.~Naud, A.~Pohl, and L.~Soares.
\newblock Fractal {W}eyl bounds and {H}ecke triangle groups.
\newblock {\em {Electron. Res. Announc. Math. Sci.}}, 26:24--35, 2019.

\bibitem{Oh}
H.~Oh.
\newblock Eigenvalues of congruence covers of geometrically finite hyperbolic
  manifolds.
\newblock {\em J. Geom. Anal.}, 25(3):1421--1430, 2015.

\bibitem{Patterson}
S.~{Patterson}.
\newblock {The limit set of a Fuchsian group}.
\newblock {\em {Acta Math.}}, 136:241--273, 1976.

\bibitem{Patterson_Perry}
S.~Patterson and P.~Perry.
\newblock The divisor of {S}elberg's zeta function for {K}leinian groups.
\newblock {\em Duke Math. J.}, 106(2):321--390, 2001.

\bibitem{Selberg_Goe}
A.~Selberg.
\newblock G{\"o}ttingen lectures.
\newblock unpublished.

\bibitem{Simon}
B.~{Simon}.
\newblock {Trace ideals and their applications. 2nd ed}.
\newblock {\em {Math. Surv. Monogr.}}, 120, 2005.

\bibitem{Sjoestrand}
J.~Sj{\"{o}}strand.
\newblock {Geometric bounds on the density of resonances for semiclassical
  problems}.
\newblock {\em {Duke Math. J.}}, 60(1):1--57, 1990.

\bibitem{Sullivan}
D.~{Sullivan}.
\newblock {The density at infinity of a discrete group of hyperbolic motions}.
\newblock {\em {Publ. Math., Inst. Hautes \'Etud. Sci.}}, 50:171--202, 1979.

\bibitem{Titchmarsh_book}
E.~C. {Titchmarsh}.
\newblock {\em {The Theory of Functions}}.
\newblock Oxford University Press, Oxford, 1958.
\newblock Reprint of the second (1939) edition.

\bibitem{Venkov_book}
A.~Venkov.
\newblock Spectral theory of automorphic functions.
\newblock {\em Proc. Steklov Inst. Math.}, 153(4), 1982.
\newblock A translation of Trudy Mat. Inst. Steklov. {{\bf{1}}53}(1981).

\bibitem{Venkov_Zograf}
A.~Venkov and P.~Zograf.
\newblock On analogues of the {Artin} factorization formulas in the spectral
  theory of automorphic functions connected with induced representations of
  {Fuchsian} groups.
\newblock {\em Math. USSR, Izv.}, 21:435--443, 1983.

\end{thebibliography}
